\documentclass[11pt]{article}
\usepackage[bb=dsserif]{mathalpha}
\usepackage{fancyhdr}   
\usepackage{comment}
\usepackage[colorlinks=true,linkcolor=blue,citecolor=blue,urlcolor=blue]{hyperref}
\fancypagestyle{plain}{
    \lhead{}
    \fancyhead[R]{\thepage}
    \fancyhead[L]{}
    
    \fancyfoot{}
}
\usepackage{bm}

\usepackage{array}
\usepackage[utf8]{inputenc} 
\usepackage[T1]{fontenc}    
\usepackage{hyperref}       
\usepackage{nicefrac} 
\newcommand{\Spvek}[2][r]{%
  \gdef\@VORNE{1}
  \left(\hskip-\arraycolsep%
    \begin{array}{#1}\vekSp@lten{#2}\end{array}%
  \hskip-\arraycolsep\right)}

\def\vekSp@lten#1{\xvekSp@lten#1;vekL@stLine;}
\def\vekL@stLine{vekL@stLine}
\def\xvekSp@lten#1;{\def\temp{#1}%
  \ifx\temp\vekL@stLine
  \else
    \ifnum\@VORNE=1\gdef\@VORNE{0}
    \else\@arraycr\fi%
    #1%
    \expandafter\xvekSp@lten
  \fi}
\usepackage{algorithm}
\usepackage{algorithmic}
\usepackage{graphicx}
\usepackage{amsfonts}
\usepackage{amsmath}
\usepackage{url}
\usepackage{amsthm}
\usepackage{amssymb}
\usepackage{enumerate}
\usepackage{lmodern}
\usepackage{multirow}
\usepackage{subcaption}
\usepackage{graphicx}
\usepackage{graphicx}
\usepackage{amsfonts}
\usepackage{amsmath}
\usepackage[utf8]{inputenc}
\usepackage{url}
\usepackage{color}
\usepackage{amsthm}
\usepackage{amssymb}
\usepackage[normalem]{ulem}
\usepackage{enumerate}
 \usepackage{paralist}
\usepackage{graphicx}

\usepackage{bbm}
\usepackage{mathtools}
\newcommand{\R}{\mathbb{R}}
\newcommand{\N}{\mathbb{N}}
\RequirePackage{amsthm,amsmath,amsfonts,amssymb}
\RequirePackage{graphicx}
\usepackage{microtype}  
\usepackage{times}
\usepackage{mathrsfs}
\usepackage{mathtools}
\usepackage{transparent}
\usepackage{graphicx}
\usepackage{wrapfig}
\usepackage{relsize,exscale}
\usepackage{url}
\usepackage{amsthm}
\usepackage{amssymb}
\newtheorem {Proposition}{Proposition}[section]
\newtheorem {Lemma}[Proposition] {Lemma}
\newtheorem {Theorem}[Proposition]{Theorem}

\newtheorem*{Theorem*}{Theorem}
\newtheorem {Corollary}[Proposition]{Corollary}

 \newtheorem{Assumption}{Assumption}
 
\def\N{\mathbb{N}}

\def\R{\mathbb{R}}

\def\z{\mathbf{z}} 
\def\v{\mathbf{v}} 
\def\x{\mathbf{x}}
\def\y{\mathbf{y}}
\def\u{\mathbf{u}}

\usepackage{xcolor}
\newcommand*{\ag}{\textcolor{magenta}}

\usepackage{relsize}
\usepackage{tikz}

\makeatletter
\def\namedlabel#1#2{\begingroup
   \def\@currentlabel{#2}%
   \label{#1}\endgroup
}
\providecommand{\diff}{\mathrm{d}}
\makeatother

\def\N{\mathbb N}

\def\x{\mathbf x}
\def\z{\mathbf z}

\def\y{\mathbf y}

\newcommand{\eps}{\varepsilon}

\DeclarePairedDelimiterX{\inner}[2]{\langle}{\rangle}{{#1},{#2}}

\makeatletter
\newcommand*{\transp}{%
	{\mathpalette\@transpose{}}%
}
\newcommand*{\@transpose}[2]{%
	\raisebox{\depth}{$\m@th#1\intercal$}%
}
\makeatother

\usepackage{minitoc}
\usepackage{enumerate}
\usepackage{bbm}
\usepackage[T1]{fontenc}    
\usepackage{lmodern}        
\usepackage{hyperref}       
\usepackage{url}            
\usepackage{booktabs}       

\usepackage{amsfonts}       

\usepackage{nicefrac}       
\usepackage{microtype}      
\usepackage{lipsum}
\usepackage{graphicx}
\usepackage{color}
\graphicspath{{./images/}}
{
     \theoremstyle{plain}
     
}
\usepackage[sectionbib,sort,compress]{natbib}
\usepackage{chapterbib}

\title{Infinitesimal behavior of Quadratically Regularized Optimal Transport and its relation with the Porous Medium Equation}
\author{
  Alejandro Garriz-Molina\thanks{Departamento de Matemática Aplicada, Universidad de Granada, Granada, Spain}
  \and
  Alberto González-Sanz\thanks{Department of Statistics, Columbia University, New York, USA.  Email: ag4855@columbia.edu}
  \and
  Gilles Mordant\thanks{Institut für Mathematische Stochastik, Universität Göttingen. Göttingen, Germany. Email: gilles.mordant@uni-goettingen.de}
}
\date{}
\begin{document}
\maketitle
 \begin{abstract}
 The quadratically regularized optimal transport problem has recently been considered in various applications where the coupling needs to be \emph{sparse}, i.e., the density of the coupling needs to be zero for a large subset of the product of the supports of the marginals. However, unlike the acclaimed entropy-regularized optimal transport, the effect of quadratic regularization on the transport problem is not well understood from a mathematical standpoint. In this work, we take a first step towards its understanding. We prove that the difference between the cost of optimal transport and its regularized version multiplied by the ratio $\varepsilon^{-\frac{2}{d+2}}$
  converges to a nontrivial limit as the regularization parameter $\varepsilon$ tends to 0. The proof  confirms a conjecture from \cite{zhang2023manifold} where it is claimed that a modification of the self-similar solution of the porous medium equation, the Barenblatt--Pattle solution, can be used as an approximate solution of the regularized transport cost for small values of $\varepsilon$.
  \end{abstract}
\section{Introduction and main result}

Optimal transport has an extremely rich history that started with Monge more than two centuries ago \citep{monge1781memoire}. The problem is, given two  positive measures $\rho_0$ and $\rho_1$ with the same finite mass defined on $\Omega_0$ and $\Omega_1$ and a cost function from $\Omega_0\times \Omega_1$ to $[0, \infty)$, to find the optimal way of transporting one measure onto the other at the minimal overall cost. That is, the optimal transport cost between two probability measures having densities $\rho_0$ and $ \rho_1$ with respect to the $d$-dimensional Lebesgue measure  $\ell_d$ is defined as 
\begin{equation}
\label{eq: OT}
 \mathcal{W}_2^2(\rho_0, \rho_1) = \inf_{\pi\in \Pi(\rho_0, \rho_1)} \int \|\x-\y\|^2 \diff\pi(\x,\y), 
 \end{equation}
where $\Pi(\rho_0, \rho_1)$ denotes the set of probability measures $\pi\in \mathcal{P}(\R^d\times \R^d)$ such that 
\[
 \int f(\x) \diff\pi(\x, \y)= \int f(\x) \diff\rho_0(\x) \quad \text{ and } \quad \int f(\y) d\pi(\x, \y)= \int f(\y) \diff \rho_1(\y) , 
 \]
for any bounded continuous function $f\in \mathcal{C}_B(\R^d)$. In the case where $\rho_1$ is absolutely continuous with respect to the Lebesgue measure $\ell_d$, the solution $\tilde{\pi}$ of \eqref{eq: OT} is unique and of the form $ \tilde{\pi} =( \nabla g\times {\rm Id})_\# \rho_1  $, where  $g$ is a convex function (see \cite{brenier1991polar,CuestaMatran}). In such a case, it is said that $\nabla g$ pushes $\rho_1$ forward to $\rho_0$.       We refer to the two monographs by \cite{Villani2003,Villani2008} for a detailed exposition.

We highlight that, in a common abuse of notation, we will refer---when ambiguity is absent---to the densities $\rho_0$ and $\rho_1$ and to the measures defined by them often by just speaking about the measures $\rho_0$ and $\rho_1$ or $\diff \rho_0$ and $\diff \rho_1$. 

After a series of fundamental theoretical developments, among which those of Kantorovich, \citet{brenier1991polar,mccann95,CuestaMatran} and many more, 
 the interest for optimal transport exploded in recent years as its usefulness for various application domains such as Machine Learning \citep{courty,DeLaraDiffeo,gordaliza2019obtaining}, Statistics \citep{Carlier2016,Hallin2020DistributionAQ,mordant2022measuring,hundrieser2023empirical,del2024nonparametric,chewi2024statistical}, Biology \citep{schiebinger2019optimal,Tameling2021,gonzalez2023two}, Astronomy \citep{Levy_2021} or Economy \citep{Galichon2016}  got acknowledged. Note that it became impossible to list all the contributions to the field.

Moving away from the classical framework, a recent line of research has been inspired by new computational developments available for a particular relaxation of the optimal transport problem. The latter relaxation consists in  adding a (convex) penalty to the initial problem given in equation~\eqref{eq: OT} above. A particular such penalty, spearheaded by \citet{cuturi2013sinkhorn} for its easy computability on GPU, relies on the addition of a Kullback--Leibler regularization term
in the initial objective function, i.e., 
$$ \mathcal{T}_{2, \eps, KL}(\rho_0, \rho_1)= \inf_{\pi\in \Pi(\rho_0, \rho_1)} \int \|\x-\y\|^2d\pi(\x,\y)+\eps\,  {\rm KL}\big(\pi\vert (\rho_0 \otimes \rho_1)\big), $$
where $(\rho_0 \otimes \rho_1)$ denotes the product measure and ${\rm KL}(\alpha \vert \beta)=\int \log\left( \diff \alpha /\diff \beta\right)\diff \alpha $ if $\alpha$ is absolutely continuous w.r.t. $\beta$  (i.e., $\alpha\ll\beta$) and ${\rm KL}(\alpha \vert \beta)=+\infty$ otherwise. Here ${\diff \alpha /\diff \beta}$ stands for the Radon–Nikodym derivative of $\alpha$ w.r.t. $\beta$. Conversely, given a positive function $\omega$ we denote as $ \omega\diff \beta $ the measure having Radon–Nikodym derivative $\omega$ w.r.t. $\beta$.
The addition of this penalty term  enables the use of a matrix-scaling algorithm: the Sinkhorn algorithm \citep{sinkhorn1964relationship}. A plethora of works and applications followed, see \citet{MAL-073}  for a book-long exposition.

Studying the infinitesimal behavior of the difference in the regularized cost  as the regularization parameter $\eps$ tends to 0, \cite{pal2019difference} proved the limit 
\begin{equation}
    \label{eq:Pal}
    \lim_{\eps\to 0^+}\frac{\mathcal{T}_{2, \eps, KL}(\rho_0, \rho_1)- \mathcal{W}_2^2(\rho_0, \rho_1)- \frac{d}{2}\,\log(\pi \eps)}{\eps}= \frac{\int \log(\rho_0) d\rho_0-\int \log(\rho_1) d\rho_1 }{2}
\end{equation}
using a Gaussian approximation of the minimizers of the regularized problem. This is a natural approximation from a heuristic point of view, given that Gaussian distributions are the fundamental solutions of the heat equation, whose Wasserstein gradient flow is over the Kullback--Leibler entropy \citep{Otto_2001}.

Just like in the heat equation, whose  solution has maximum support for any positive time, the solutions of the optimal transport problem regularized by logarithmic entropy share the same property. That is, for every $\eps>0$, the resulting coupling  is  dense, i.e., each point on the product space $\Omega_0 \times \Omega_1$ has a positive density. On the contrary, the (unregularized) optimal transport plan for the squared Euclidean distance as a cost is concentrated on the gradient of a convex function.   To remedy this, \citet{blondel2018smooth} introduced the Quadratically Regularized Optimal Transport (QOT)
\begin{equation}
    \label{primalQuad}
    \mathcal{T}_{2, \varepsilon, (\cdot)^2}(\rho_0, \rho_1) = \inf_{\pi\in \Pi(\rho_0, \rho_1)} \int \|\x-\y\|^2\diff \pi(\x,\y)+\eps  \left\| \frac{\diff \pi}{  \diff (\rho_0 \otimes \rho_1) }\right\|_{L^2( \rho_0 \otimes \rho_1)}^2,
\end{equation}  
where  
$$ \left\| \frac{\diff \pi}{  \diff (\rho_0 \otimes \rho_1) }\right\|_{L^2( \rho_0 \otimes \rho_1)}^2=\int \left( \frac{\diff \pi}{  \diff (\rho_0 \otimes \rho_1)}\right)^2 \diff (\rho_0 \otimes \rho_1) $$
in the case where $ {\diff \pi}\ll{  \diff (\rho_0 \otimes \rho_1) }$, and $+\infty$ otherwise. 
Applications of QOT have since then started blossoming. Examples include \citet{essid2018quadratically,zhang2023manifold, van2023optimal,mordant2023regularised}. For some theoretical results we refer to \cite{nutz2024quadratically}. \citet{blondel2018smooth} observed that the minimizer of \eqref{primalQuad} is  {\it ``sparse''} in the sense that its topological support seems to increase progressively as $\eps$ gets larger. 
This increase of a sparse support suggests a diffusion behavior like that of the Porous Medium Equation (PME) where epsilon would play the role of the time parameter. For this reason, before setting this parallel on solid ground, we provide a short introduction to the PME in the next section.

\subsection{On the Porous Medium Equation and the QOT }

The Porous Medium Equation posed in the whole space reads, for an initial datum $u_0$,
\[
\begin{cases}
    \partial_t u(t,\x) =\Delta u(t,\x)^m,\quad & t>0,\ \x\in \mathbb{R}^d\\
    u(0,\x)=u_0(\x) & x\in \mathbb{R}^d,
\end{cases}
\]
where $m>1$ is the non-linearity exponent. A mathematical study for this equation appeared first, in the case $m=2$, in the study of groundwater infiltration by~\cite{boussinesq_1903} and since then it has attracted the attention of numerous experts. The general case $m>1$ can be derived from the study of an an-isotropic gas in a porous medium (hence the name) subject to Darcy's Law, see the seminal (and independent of each other) works of~\cite{leibenzon_1930} and~\cite{muskat_1937}. Since then, it has enjoyed a continuous attention from the mathematical community due to its applications in many fields, like plasma radiation, biology, population theory or even cosmology. We refer the interested reader to the book of~\cite{vazquez_2007}, broadly considered as the go-to reference on this matter. Therein, the reader can find a more detailed explanation of the main characteristics of the PME that are relevant for the study of the QOT. Even though in this article we focus on the case $m=2$, we will discuss the general case $m>1$.

The first property that leaps to the eye is the fact that the PME is non-linear, i.e., the sum of two solutions of the equation is not a solution itself. This makes it impossible for the existence of a linear description of its solutions similar to the heat equation, where the solution can be expressed as the convolution of the initial datum $u_0$ and the fundamental solution, the self-similar Gaussian profile (also referred as the kernel of the equation). The solution of the PME cannot be expressed as the convolution of the initial datum  with any kernel but, nevertheless, there exists a fundamental solution $\mathcal{B}(t,x)$ of the PME, in the sense that
\[
\begin{cases}
    \partial_t \mathcal{B}(t,\x) =\Delta \mathcal{B}(t,\x)^m,\quad & t>0,\ x\in \mathbb{R}^d\\
    \mathcal{B}(0,\x)=\delta_0(\x) & \x\in \mathbb{R}^d,
\end{cases}
\]
where $\delta_0(\x)$ is Dirac's Delta function located at $\x=0$ and the previous equalities should  be understood in the weak sense (or integral sense). This fundamental solution is called Barenblatt or Barenblatt--Kompaneets--Zeldovich profile due to its discoverers~\cite{barenblatt_1952} and~\cite{zeldovich_kompaneets_1903}. It reads
\begin{equation}\label{eq:barenblatt_profile}
\mathcal{B}(t,\x)=\frac{1}{t^\alpha}\left[C- \beta\frac{m-1}{2m}\frac{\|\x\|^2}{t^{2\beta}} \right]_+^{\frac{1}{m-1}},\quad \text{where}\quad \alpha=\frac{d}{d(m-1)+2},\  \beta= \frac{1}{d(m-1)+2},
\end{equation}
and $C>0$ is a free constant that dictates the mass of the profile.
Contrarily to the Kullback--Leibler regularisation case, we cannot work with linear tools, i.e., semigroup theory. Still,  the profile of the solution above will play a key role in our work, see formula~\eqref{relationPrimalDual} to mention but one example. 

The next property that is relevant in our study is the \textit{slow} nature of the diffusion in the PME. If one writes the equation in terms of its diffusivity $D(u)$ then we see that
\[
\partial_t u={\rm div}\left(D(u)\cdot \nabla u \right),\quad \text{with}\quad D(u)=m u^{m-1}
\]
and therefore the diffusivity disappears whenever $u=0$; in other words, our equation is degenerate parabolic at the points where $u=0$. This provokes the phenomena known as \textit{slow diffusion}, where every solution whose initial datum has compact support will maintain a compact support that grows over time, in sharp contrast with the heat equation, where every solution becomes positive everywhere at any positive time. This characteristic of the PME relates closely to the \textit{sparse} nature of the minimizer in the QOT observed by~\citet{blondel2018smooth}.

Finally, there is one last connection between the PME and the QOT that we would like to mention.~\cite{Otto_2001} explained how that the PME could be understood as a gradient flow with the 2-Wasserstein metric of the functional formed by the second moment of the solution and its free energy, namely
\begin{equation}\label{eq:entropy}
E(u)=\int \frac{1}{2}\|\x\|^2u(\x) + \frac{1}{m-1}u(\x)^m\ {\rm d}\ell_d(\x),
\end{equation}
see also~\cite{zhang2023manifold} for some applications. More or less at the same time~\cite{Carrillo_Toscani_2000_AsymptoticLO} followed this approach, using $E(u)$ as the natural entropy of the system, to describe the large-time behavior of the solutions of the problem. Note that  the resemblance between~\eqref{primalQuad} and ~\eqref{eq:entropy} in the case $m=2$. Quadratically regularized optimal transport thus mimics the minimisation of the free energy just defined under the coupling constraints. As we shall see later, this last observation is where the geometry of unregularized optimal transport comes into the picture and plays a crucial role.

\subsection{Assumptions and main results}

The existence of a potential link between QOT and the porous medium equation was considered folklore in the community in the last years even though never rigorously established. The first developments clearly fleshing this out were proposed by \cite{zhang2023manifold}. Their work further suggests that, by using Barenblatt profile approximations instead of Gaussian approximations (for the case of the Kullback--Leibler regularisation studied by \cite{pal2019difference}), one could characterize the non-trivial limit of 
\[
r(\eps)  \big(\mathcal{T}_{2, \varepsilon, (\cdot)^2}(\rho_0, \rho_1) -\mathcal{W}_2^2(\rho_0, \rho_1)  \big)
\]   
for a proper choice of $r(\eps)$. Via a proof   based on quantization, \cite{Eckstein2023} determined the value of the ratio $r(\varepsilon)$ as $r(\varepsilon) = \varepsilon^{-\frac{2}{d+2}}$. It is worth highlighting, however, that the nonlinear nature of the penalty in~\eqref{primalQuad} makes impossible to utilize the linear tools present in~\cite{pal2019difference} or most of the techniques present in the existing literature. The method we propose here is novel and relies mostly on the understanding of Barenblatt-like densities, inspired by~\eqref{eq:barenblatt_profile}, and on the clever usage of mass transportation theory. 

In this work, we prove the complete (implicit) conjecture of \cite{zhang2023manifold} with this ratio under the following assumptions. 
\begin{Assumption}\label{Assumptions}
We assume the following: 
\begin{enumerate}
    \item $\Omega_1$ and $\Omega_0$ be open bounded sets such that 
    there exists $\delta_0\in (0, \pi)$ and $\sigma>0$ such that for  every $i\in \{0,1\}$ and $\x\in \Omega_i$ there exists a cone
    $$\mathcal{C}_{\x, {\bf v},\delta_0, \sigma}=\left\{ \x+{\bf u}: \|{\bf v}\|\|{\bf u}\| \cos\left(\frac{1}{2}\theta\right) \leq \langle {\bf v} ,  {\bf u }
 \rangle \leq \|{\bf v}\|\delta_0  \right\} $$ 
    with vertex $\x$,  height $\delta_0$ and angle $\theta$ such that $\mathcal{C}_{\x, {\bf v}\delta_0, \sigma}\subset \Omega_i$.
    
    \item The functions $\rho_0\in \mathcal{C}^{\beta}(\bar{\Omega_0})$ and $\rho_1\in \mathcal{C}^{\beta}(\bar{\Omega_1})$ for some $\beta>0$ are densities defining  probability measures 
    $$ 1=\int_{\Omega_i} \rho_i(\x) \diff\ell_d(\x)=\int d\rho_i(\x)=\rho_i(\Omega), \quad i=1, 2.  $$
    \item Such  densities are upper lower bounded in their domain, i.e.,  
\begin{equation}\label{eq:UpperLower0}
    \lambda  \leq \rho_0(\x) \leq \Lambda \ {\rm and} \ \lambda  \leq \rho_1(\y) \leq \Lambda \quad \text{for all} \ \x\in \Omega_0 \ {\rm and} \ \y\in \Omega_1
\end{equation}
for some $0< \lambda \leq \Lambda < \infty$.
\item The unique gradient of a convex function $g$ pushing $\rho_1$ forward to $\rho_0$ belongs to $\mathcal{C}^{2, \alpha}(\Omega_1)$ and 
$$ \sigma_m(g)\, {\rm Id} \leq \nabla^2 g \leq  \sigma_M(g)\, {\rm Id} \quad {\rm in }\ \Omega_1  $$
for some $0<\sigma_m(g) \leq \sigma_M(g)< \infty$.
\end{enumerate}
\end{Assumption}
Note that Assumption~\ref{Assumptions}, 4. implies also that the convex conjugate of $g$, namely
$$ g^*(\x)=\sup_{\y\in \R^d}\{ \langle \x, \y\rangle - g(\y) \},$$
is $ \mathcal{C}^{2}(\Omega_0)$. Moreover, it is well known that $\nabla g^*$ pushes $\rho_0$ forward to $\rho_1$.  

The following is our main result. 


\begin{Theorem}\label{Theorem:Main}
    Let Assumption \ref{Assumptions} hold and $g$ be the unique function such that $\nabla g$ pushes $\rho_1$ forward to $\rho_0$.   Then it holds that
    $$ \lim_{\varepsilon\to 0^+}  \frac{\mathcal{T}_{2, \varepsilon, (\cdot)^2}(\rho_0, \rho_1) -\mathcal{W}_2^2(\rho_0, \rho_1) }{\varepsilon^{\frac{2}{d+2}}}  = {\frac{d^{\frac{d+4}{d+2}}(d+2)^{\frac{2}{d+2}}}{\left(\mathcal{H}^{d-1}(\mathcal{S}^{d-1}) \right)^{\frac{2}{d+2}}}}   \int_{\Omega_0} \big(\rho_0(\x)\rho_1[\nabla g^*(\x)]\big)^{-\frac{1}{(d+2)}}\diff \rho_0(\x),$$
\end{Theorem}

As a consequence of this result, we can write explicitly the relationship between solutions of the PME and the optimizers of the QOT. Define the energy 
  \[E_\eps(\pi):=   \int \frac{1}{2} \|\x-\y\|^2 \pi(\x,\y)+{\varepsilon} \cdot\pi^2(\x, \y) \diff (\rho_0 \otimes \rho_1)(\x,\y)
  \]
  and let $g$ be the unique function such that $\nabla g$ pushes $\rho_1$ forward to $\rho_0$. Define then, for any $\x'\in\Omega_0$,
  \[
  \varrho(\x'):=\rho_0\left( \left[\nabla^2g^*(\x')\right]^{-\frac{1}{2}}\x'\right)\cdot  \rho_1\left(\nabla g^*\left( \left[\nabla^2g^*(\x')\right]^{-\frac{1}{2}}\x'\right)\right).
  \]

Choose now any $\x'\in\Omega_0$ and let $u(t, \x; \x')$ be the solution of 
  \[
\begin{cases}
    \partial_t u(t,\x;\x') =\frac{1}{2(d+2)}\Delta_{\x} u(t,\x; \x' )^2,\quad & t>0,\ \x\in \Omega_0\\[10pt]
    u(0,\x; \x')=\varrho(\x')^{-\frac{1}{d+2}} \cdot \delta_{\x'}(\x) & \x\in \mathbb{R}^d,
\end{cases}
\]
where $\Delta_{\x}$ represents the Laplacian only in the $\x$'s variable, $\delta_{\x'}(\x)$ denotes the Dirac's Delta function on the point $\x=\x'$ and the previous solution must be understood in the weak sense. The constant $1/2(d+2)$ in front of the Laplacian operator can be considered just as a scaling factor in the $\x$ variable. The formula for $u$ reads
\[
u(t,\x;\x')=\frac{1}{t}\left(\frac{C_d\  t^{\frac{2}{d+2}}}{\varrho(\x')^{\frac{1}{d+2}}} -\frac{1}{2}\|\x-\x'\|^2 \right)_+
\]
where $C_d$ is just a normalizing constant depending only on the dimension of the space that will be stated later.

Next, to each $\x'\in\Omega_0$ and each solution $u$ of the previous PME it corresponds a transformed function
  \[
  \begin{aligned}
  v(t,\x;\x')&:= u\left(t, \left[\nabla^2g^*(\x')\right]^{\frac{1}{2}}\x; \left[\nabla^2g^*(\x')\right]^{\frac{1}{2}}\x'\right)\\[10pt]
  &= \frac{1}{t}\left(\frac{C_d\  t^{\frac{2}{d+2}}}{(\rho_0(\x')\rho_1(\nabla g^*(\x')))^{\frac{1}{d+2}}} -\frac{1}{2}\|\x-\x'\|^2_{\nabla^2 g^*(\x')} \right)_+,
  \end{aligned}
  \]
  where
  $$ 
    \|\x-\x' \|^2_{\nabla^2 g(\x)}= \langle \x-\x',  \nabla^2 g(\x)(\x-\x')\rangle, 
  $$ 
We are ready now to state the corollary.
\begin{Corollary}
    Under the assumptions of Theorem~\ref{Theorem:Main} let $\pi_{(\eps)}$ be the minimizer of \eqref{primalQuad} for $\eps>0$ and. Then it holds 
$$ \lim_{\varepsilon\to 0^+}  \varepsilon^{-\frac{2}{2+d}} \vert E_\eps(\pi_{(\eps)}(\x,\y))-E_\eps(v(\eps, \x, \nabla g(\y)))\vert=0 $$
\end{Corollary}

We sketch a proof of the corollary for the interested reader. The details are omitted, as they can be readily deduced by repeating the  calculations provided in the proof of Theorem~\ref{Theorem:Main}.  Later on we will define a family of approximations $\pi_\varepsilon$ to the true minimizer $\pi_{(\varepsilon)}$. The key of these approximations is that, as $\varepsilon\to 0$,
\begin{equation}\label{eq:pi_v}
\pi_\varepsilon(\x,\y)\sim v(\varepsilon,\x;\nabla g(\y)).
\end{equation}
The idea consists then in writing, adding and subtracting terms,
\[
\begin{aligned}
\varepsilon^{-\frac{2}{2+d}}\left[ E_\eps(\pi_{(\eps)}(\x,\y))-E_\eps(v(\eps, \x, \nabla g(\y)))\right] &=\underbrace{\varepsilon^{-\frac{2}{2+d}}\left[ E_\eps(\pi_{(\eps)}(\x,\y))- \mathcal{W}^2_2(\rho_0,\rho_1)\right]}_{\mathcal{I}}\\[10pt]
&+ \underbrace{\varepsilon^{-\frac{2}{2+d}}\left[ \mathcal{W}^2_2(\rho_0,\rho_1)-E_\eps(\pi_{\eps}(\x,\y))\right]}_{\mathcal{II}}\\[10pt]
&+\underbrace{\varepsilon^{-\frac{2}{2+d}}\left[ E_\eps(\pi_{\eps}(\x,\y))-E_\eps(v(\eps, \x, \nabla g(\y)))\right]}_{\mathcal{III}}.
\end{aligned}
\]
Then thanks to Theorem~\ref{Theorem:Main} and the fact that $\pi_\varepsilon$ approximates $\pi_{(\varepsilon)}$ we have $\mathcal{I}\to -\mathcal{II}$, and then from~\eqref{eq:pi_v}, applying the same ideas as in the proof of Theorem~\ref{Theorem:Main}, we obtain that $\mathcal{III}\to 0$.

Apart from obtaining the limit conjectured by \cite{zhang2023manifold}, the main interest of this article consists on providing a constructive approach to the study of such limits for more general penalties. Our technique should work with very little changes for penalties of the form $  \left\| \cdot \right\|_{L^m( \rho_0 \otimes \rho_1)}^m$ for $m>1$, which corresponds to the full range of parameters of the Porous Medium Equation. We leave this generalization for further work and we highlight that we do not know if this technique could work for non-convex penalties of the form $  \left\| \cdot\right\|_{L^m( \rho_0 \otimes \rho_1)}^m$ for $0<m<1$. This approach differs greatly from the one taken by \cite{pal2019difference}, which relies deeply on the linearity of the heat equation.   In addition, our analysis further  elucidates why the correct rate of convergence in Theorem~\ref{Theorem:Main} is $\eps^{-\frac{2}{d+2}}$.

To prove Theorem \ref{Theorem:Main}, we need to provide both an upper bound and a lower bound for $\mathcal{T}_{2, \varepsilon, (\cdot)^2}(\rho_0, \rho_1)$. To provide the lower bound, we will use the dual formulation (see \cite{nutz2024quadratically})
\begin{multline}
    \label{dualQuadr1}
    \mathcal{T}_{2, \varepsilon, (\cdot)^2}(\rho_0, \rho_1) ={2\cdot} \sup_{a,b\in L^2(\rho_0)\times L^2(\rho_1)} \int \bigg\{a(\x) + b(\y) \\
    -\frac1{{2\eps}}\left(a(\x)+b(\y)-\frac{1}{{2}}\|\x-\y\|^2\right)_+^2\bigg\} \diff\rho_0(\x) \diff \rho_1(\y)
\end{multline}
of \eqref{primalQuad}, whose solutions  are the so-called regularized potentials. 
The solutions of the primal and dual problems are related in the following way;  $(a_\eps, b_\eps)$ is a maximizer of \eqref{dualQuadr1} if and only if 
\begin{equation}
\label{relationPrimalDual}
   \frac{1}{\eps}\left(a_\eps(\x)+b_\eps(\y)-\frac{1}{{2}}\|\x-\y\|^2  \right)_+\diff\rho_0(\x) \diff \rho_1(\y)
\end{equation}
is a minimizer of \eqref{primalQuad}. The reader should note the resemblance between~\eqref{relationPrimalDual} and~\eqref{eq:barenblatt_profile} in the case $m=2$, with $a_\varepsilon(\x), b_\varepsilon(\y)\sim C\varepsilon^{\frac{2}{d+2}}$ and $t=\varepsilon$. This relation suggests as a candidate  a pair of functions  \( (\tilde{f}_\eps, g) \) such that 
\begin{align*}
    &\int \left(\tilde{f}_\eps(\x)+\frac{1}{2}\|\y\|^2 -g(\y)-\frac{1}{{2}}\| \x-\y\|^2 \right)_+ d\rho_1(\y)= \eps \quad \quad \text{for all } \x \in \Omega_0.
\end{align*}
After some calculations, 
Lemma~\ref{LemmaDiver} yields 
$\tilde{f}_\eps(\x)\asymp  \frac{1}{2}\|\x\|^2 -g^*(\x)+ C_\eps(\x),  $
where $$ C_\eps(\x):=\frac{\eps^{\frac{2}{d+2}}}{C_d^{\frac{2}{d+2}} \big(\rho_0(\x)\rho_1[\nabla g^*(\x)]\big)^{\frac{1}{d+2}}}\quad {\rm for}\quad  C_d:= 2^{\frac{d+2}{2}} \mathcal{H}^{d-1}(\mathcal{S}^{d-1}) \frac1{d(d+2)}. $$
Hence, we will find the lower bound by finding the limit as $\eps\to 0$ of 
\begin{equation}
    \label{lowerMultby2}
    {2\cdot } {\eps^{-\frac{2}{d+2}}}\left({\Gamma_\eps\left(f_\eps, \frac{1}{2}\|\cdot\|^2 -g\right)- {\frac{1}{2}}\mathcal{W}_2^2(\rho_0, \rho_1)} \right)
\end{equation}
where 
    $ {f}_\eps(\x)=  \frac{1}{2}\|\x\|^2 -g^*(\x)+ {C_\eps(\x)}$ and 
    $$  \Gamma_\eps (a, b)=\int \bigg\{a(\x) + b(\y) 
    -\frac1{2\varepsilon}\left(a(\x)+b(\y)-\frac{1}{2}\|\x-\y\|^2\right)_+^2\bigg\} \diff\rho_0(\x) \diff \rho_1(\y). $$
For the upper bound, we will find a coupling $\pi_\eps\in \Pi(\rho_0, \rho_1)$ such that the functional 
$$ \mathbb{H}_\eps(\pi)=   \int \|\x-\y\|^2d\pi(\x,\y)+{\eps} \left\| \frac{d\pi}{  d(\rho_0 \otimes \rho_1) }\right\|_{L^2( \rho_0 \otimes \rho_1)}^2, $$
at $\pi_\eps$ achieves the same limit. The natural candidate is the measure $\tilde{\pi}$ with  \( \frac{\diff \tilde{\pi}}{\diff (\rho_0\otimes \rho_1)} (\x, \y) = \frac{1}{\eps}(C_\eps(\x)-D(\x, \y))_+ \), 
where 
\begin{equation}\label{divergence}
    D(\x, \y )= g^*(\x)+g(\y)-\langle \x, \y\rangle
\end{equation}
is the  well-known  Bregman divergence. 
However, \( \tilde{\pi} \) is neither a coupling nor a probability measure in the product space. 
Therefore, it must be adjusted using the following procedure: Firstly, transport \(\rho_0\) to \(\rho_1\) to work exclusively with the measure
\[ 
\frac{1}{\varepsilon}(C_\varepsilon(\mathbf{x}) - D(\mathbf{x}, \nabla g^*(\mathbf{x}')))_{+} \, \mathrm{d}\rho_0(\mathbf{x}) \, \mathrm{d}\rho_0(\mathbf{x}'). 
\]
Since $D(\mathbf{x}, \nabla g(\mathbf{x}'))$ is more difficult to control than its quadratic approximation (see Lemma~\ref{LemmaDiver})  $$ 
\frac{1}{2}\|\x-\x' \|^2_{\nabla^2 g(\x)}=\frac{1}{2} \langle \x-\x',  \nabla^2 g(\x)(\x-\x')\rangle, 
$$ 
we construct the symmetric version of its density
\[ 
m_\varepsilon (\mathbf{x}, \mathbf{x}') = \frac{1}{2 \varepsilon} \left( C_\varepsilon(\mathbf{x}) + C_\varepsilon(\mathbf{x}') - \frac{1}{2}\|\x-\x' \|^2_{\nabla^2 g(\x)}-\frac{1}{2}\|\x-\x' \|^2_{\nabla^2 g(\x')} \right)_{+}.
\]
Again, we highlight the resemblance between this last definition and~\eqref{eq:barenblatt_profile} in the case $m=2$, with $C_\varepsilon(\x), C_\varepsilon(\y)\sim C\varepsilon^{\frac{2}{d+2}}$, also $\|\x-\x' \|_{\nabla^2 g(\x)},\|\x-\x' \|_{\nabla^2 g(\x')}\sim \|\x-\x'\|$ (see Lemma~\ref{LemmaDiver} later on) and finally $t=\varepsilon$.

Subsequently, we normalize to create the measure with density 
\[ 
\xi_{\varepsilon} (\mathbf{x}, \mathbf{x}') = \frac{m_\varepsilon(\mathbf{x}, \mathbf{x}')} {\int m_\varepsilon(\mathbf{x}, \mathbf{x}') \, \mathrm{d}(\rho_0 \otimes \rho_0)(\mathbf{x}, \mathbf{x}')}
\] 
with respect to $\rho_0 \otimes \rho_0$.
Within the product space, the symmetry $\xi_{\varepsilon}$ is important because it implies   that both its marginals are the same and with  density \( \rho_\eps(\x')=\int {\xi}_{\eps} (\x, \x') \diff \rho_0(\x)\) with respect to $\rho_0$. Then, transport \(  \rho_\varepsilon \diff \rho_0 \) to \( \diff \rho_0 \) and define the probability measure 
\[ 
\mathrm{d} \mu_\varepsilon (\mathbf{x}, \mathbf{y}) = \frac{\xi_{\varepsilon}(\nabla \phi_\varepsilon^*(\mathbf{x}), \nabla \phi_\varepsilon^*(\mathbf{x}'))} {\rho_\varepsilon(\nabla \phi_\varepsilon^*(\mathbf{x})) \rho_\varepsilon(\nabla \phi_\varepsilon^*(\mathbf{x}'))} \mathrm{d}\rho_0(\mathbf{x}) \, \mathrm{d}\rho_0(\mathbf{x}'), 
\]
where \(\nabla \phi_\varepsilon\) is the optimal transport map between  \(\diff \rho_0\) and \(\rho_\varepsilon \diff \rho_0\). For the rates of this coupling to be appropriate, it is necessary that the Lipschitz constant of \(\nabla \phi_\varepsilon\) tends to 1 uniformly (see \eqref{stabilty}). Using Caffarelli's regularity theory (see \cite{Caffarelli1992,Caffarelly1990}), one can obtain the norm \(\| \phi_\varepsilon \|_{\mathcal{C}^{2, \alpha}(\Omega_0^{\delta})}\), where 
\[
\Omega_0^{\delta} = \{ \mathbf{x} \in \Omega_0 : \mathrm{dist}(\mathbf{x}, \partial \Omega_0) \geq \delta \} 
\] 
and \(\delta > 0\), uniformly bounded in \(\varepsilon\). In one dimension or in the flat torus, this holds for \(\delta = 0\), and the proof concludes at this step. In higher dimensions and with Euclidean cost, it is not entirely clear that this bound has been established in the literature---regularity up to the boundary is known for strongly convex and smooth domains  \citep{Urbas1997,Caffarelli1996Bound2}, but the dependence on the constants is not clear to us. The same issue appears in \cite{manole2021plugin} in a different context. Instead of attempting to prove this result, we will circumvent the issue by dividing the space $\Omega_0$ into two parts: points that are less than $\delta$ away from the boundary and those that are farther. For the latter, we can apply interior regularity, where the dependence on the constants is explicit. For points near the boundary, we will use a different coupling---the so-called {\it ``frame coupling''}---that does not achieve the appropriate limit but does not cause the functional to diverge. After taking the limit as  $\eps$ tends to 0, we will then take the limit as $\delta$ tends to 0 and check that it matches the lower bound.

As Caffarelli's results are only valid for convex domains, the previous discussion is applicable only to convex domains. 
The final step of the proof is to remove that assumption. It involves creating a tessellation of the domain \(\Omega_0\) into disjoint squares \(\{I_i\}_{i=1}^m\) and constructing a frame coupling in \(\Omega_0 \setminus \{I_i\}_{i=1}^m\), applying the previous coupling to each of the squares. This piecewise-defined coupling is referred to as the {\it ``stained glass"} coupling (see Figure~\ref{fig:images}).

The rest of the paper is organized as follows. In Section \ref{section:notation}, we introduce the notation that will be utilized throughout the paper.  In Section~\ref{section:OT}, we present the necessary results from optimal transport theory that are essential for proving the main theorem. In Section \ref{Section:lower}, we will provide the lower bound, and in Section \ref{section:upper}, we will establish the upper bound. As mentioned earlier, the upper bound will be proven first for convex supports (Section~\ref{section:convex}) and then for arbitrary supports (Section~\ref{section:General}). 

\section{Notation and auxiliary results}\label{section:notation}
In order to present the foundational concepts more seamlessly, we begin by defining a finite positive constant $C(a_1, \dots, a_d)$ that depends solely on a set of parameters $(a_1, \dots, a_d)$. For the sake of brevity, we denote a constant that depends only on $\rho_0$, $\rho_1$, $\Omega_0$, $
\Omega_1$ and $g$ by $C$, specifically $C = C(\rho_0, \rho_1, g)$. A function $v$ of $\varepsilon$ is said to be a little-o of another function $u$ (denoted as $v = o(u)$) if $\frac{v(\varepsilon)}{u(\varepsilon)} \to 0$ as $\varepsilon \to 0$. Occasionally, the function $v(\varepsilon, \mathbf{x})$ may depend on $\mathbf{x}$ within $\Omega_0$. In such cases, we use the notation $v = o(u)$ if $\sup_{\mathbf{x} \in \Omega_0} \frac{v(\varepsilon, \mathbf{x})}{u(\varepsilon)} \to 0$ as $\varepsilon$ approaches zero. In the same manner, we will use the standard big-O notation.
 The notation $x \asymp_{\kappa} y$ is used to express an equivalence of order of convergence, considering multiplicative constants that depend solely on a parameter $\kappa > 0$, particularly in contexts approaching the limit as $\varepsilon \to 0$. As before $\kappa=(\rho_0, \rho_1, \Omega_0,
\Omega_1, g)$ we just write $x \asymp  y$.

Moreover, within a compact set $\Omega$, the space $\mathcal{C}^{p, \alpha}(\Omega)$ is the space of H\"older continuous functions of order $p$ with H\"older continuity parameter $\alpha$ in the open interval $(0,1)$. The norm for this function space is defined as 
\[
\|f\|_{\mathcal{C}^{p, \alpha}(\Omega)} = \sum_{k=0}^{p}\sum_{|\alpha|=k} \sup_{x \in \Omega} \left| \frac{d^k}{dx_{a_1} \cdots dx_{a_k}} f(x) \right| + \sum_{|\alpha|=p} \sup_{\substack{x, y \in \Omega \\ x \neq y}} \frac{|f(x) - f(y)|}{\|x - y\|^\alpha}.
\]

Additionally, the notation $\ell_d$ is employed to denote the Lebesgue measure, and $g^*$ refers to the convex conjugate of $g$, calculated as $g^*(x) = \sup_{y \in \mathbb{R}^d} (\langle x, y \rangle - g(y))$. The Hausdorff measure of dimension $d-1$ is represented by $\mathcal{H}^{d-1}$.

Furthermore, for a positive definite matrix $A$, its square root, denoted $A^{1/2}$, is the unique positive definite matrix $V$ such that $V^2 = A$.  The product measure $\rho_0 \otimes \rho_1$ and the space $L^2(\mu)$, which includes functions $f$ with a finite $L^2(\mu)$-norm given by
\[
\|f\|_{L^2(\mu)} = \left( \int f^2(x)  \diff \mu(x) \right)^{1/2},
\]
are also crucial. In instances where Borel measures $\mu$ and $\nu$ satisfy $\mu \ll \nu$, the notation $\frac{d\mu}{d\nu}$ represents the Radon-Nikodym derivative of $\mu$ with respect to $\nu$.  For the particular cases of $\rho_i$, for $i=1,2$, to avoid complicating the notation, we will use the following abuses of notation: $\frac{d\rho_i}{d\ell_d}= \rho_i $ and 
$$ \int f(\x) d\rho_i(\x) = \int_{\Omega_i} f(\x) \rho_i(\x)  \diff \ell_d(\x), \quad i=1, 2.  $$
For a Borel set $A$,   $\chi_A$ denotes the indicator function of $A$. Sometimes, when it can cause confusion, we write $\chi_{[A]}$. We say that a function $T$ pushes a measure $\mu$ forward to another measure $\nu$ if $\nu(A)=\mu(T^{-1}(A))$ for all Borel set $A$. In such a case, we write $T_\# \mu=\nu$.  A set $A$ is compactly contained in $B$ ($A\subset \subset B$)  if there exists a compact set $K$ and an open set $\mathcal{U}$ such that $A\subset K\subset \mathcal{U}\subset  B$. The Euclidean Ball with center \(\x\) and radius \(\delta > 0\) is denoted as \(\mathbb{B}(\x, \delta)\).


\subsection{Optimal transport background}\label{section:OT}
The optimal transport problem  \eqref{eq: OT} admits the dual formulation
\begin{equation}
    \label{eq:OTdual}
    \mathcal{W}_2^2(\rho_0, \rho_1) 
   ={2\cdot}\sup_{(a,b)\in \Phi} \int a(\x) \diff\rho_0(\x) + \int b(\y) \diff\rho_1(\y),
\end{equation}
where 
$ \Phi=\{ (a, b)\in (\mathcal{C}(\R^d))^2: \ a(\x)+b(\y)\leq \|\x-\y\|^2/2, \ \forall \x, \y \in \R^d\} .$
Under Assumption~\ref{Assumptions}, the solution of the primal problem \eqref{eq: OT} is given by $({\rm Id}\times \nabla g)_\# \rho_1$ and the solution of \eqref{eq:OTdual} by the pair 
\begin{equation}
    \label{dualrelationOT}
    \left( \frac{1}{2}\|\cdot\|^2 -g^*,  \frac{1}{2}\|\cdot\|^2 -g \right).
\end{equation}
In general, the function $g$ solves the Monge-Ampère equation
$$ \det(\nabla^2 g)=\frac{\rho_1}{\rho_0(\nabla g) } \quad {\rm in} \ \Omega_1, \quad \nabla g (\Omega_1)= \Omega_0  $$
in the so-called Brenier sense (see \cite{Fi}). In our case, by assumption, $g$ solves it in the strong sense. The Monge-Ampère equation is a degenerate elliptical nonlinear differential equation. The regularity of convex solutions has been studied by Caffarelli in \cite{Caffarelli1992,Caffarelly1990}. Here, we highlight the result that we will use in this work.
\begin{Theorem}[Caffarelli]\label{Theorem:Caffarelli}
   Let $\Omega$ be open and convex $q_1,q_2\in \mathcal{C}^{0,\alpha}(\Omega)$ be positive functions such that 
    $$ \int_{\Omega} q_1\, d\ell_d= \int_\Omega q_2\,  d\ell_d=1. $$
    Then there  exists a
    unique (up to additive constant) convex solution $u$ of the boundary value problem 
    $$ \det(\nabla^2 u  )=\frac{q_1}{q_2(\nabla u )}\quad {\rm in} \ \Omega, \quad \nabla u (\Omega)= \Omega $$
    in $\mathcal{C}^{2,\alpha}_{Loc}(\Omega)$.
    Moreover, for each $K\subset \subset \Omega$ 
    $$ \|u\|_{\mathcal{C}^{2, \alpha}(K)} \leq  C\left(K, \inf_{\x\in \Omega} q_1, \inf_{\x\in \Omega} q_2, \|q_1\|_{\mathcal{C}^{0,\alpha}(K)}, \|q_2\|_{\mathcal{C}^{0,\alpha}(K)} \right) $$
\end{Theorem}
The assumption of regularity on $g$ is key in the infinitesimal developments we are undertaking. Lemma~\ref{LemmaDiver} is fundamental as it relates $D(\x, \y)$ with its linearized inner product 
 $$ 
 \langle \x, \z \rangle_{\nabla^2 g^*(\x)}= \langle \x, \nabla^2 g^*(\x) \z \rangle.
 $$ 
 We omit its proof as it is a mere second-order Taylor expansion of the Bregman divergence 
$$D(\x, \y)= g^*(\x)+g(\y)-\langle \x, \y\rangle $$
of $g$.

\begin{Lemma}\label{LemmaDiver}
    Let Assumption~\ref{Assumptions} hold. Then there exists $\sigma_M(g)\geq \sigma_m(g)>0$ such that \begin{align}\label{quadraticBoundsDiv}
      \frac{\sigma_m(g)}{2} \| \x-\x'\|^2&\leq  D(\x, \nabla g^*(\x'))\leq  \frac{\sigma_M(g)}{2} \| \x-\x'\|^2.
    \end{align} 
Moreover,  there exists $C>0$ such that 
\begin{equation}
 \left\vert    D(\x, \nabla g^*(\x'))-\frac{1}{2}\|\x-\x' \|^2_{\nabla^2 g^*(\x)}\right\vert \leq  C\|\x-\x'\|^{2+\alpha}, \label{DivXb}
\end{equation}   
where 
$$ \frac{1}{2}\|\x-\x' \|^2_{\nabla^2 g^*(\x)}=\frac{1}{2} \langle \x-\x',  \nabla^2 g^*(\x)(\x-\x')\rangle$$
and $\alpha$ is as in Assumption~\ref{Assumptions}. 
\end{Lemma}

As a direct consequence of Lemma~\ref{LemmaDiver} we obtain the following.

 \begin{Lemma}\label{LemmaDevelopmentLeps}
     Let Assumption~\ref{Assumptions} hold. Then the following estimates hold: 
     \begin{enumerate}
         \item There exists a constant $C $ and $\eps_0>0$  such that  
    $$ \left\{ (\x, \x')\in \Omega_0\times \Omega_0: \    C_\eps(\x)\geq  D(\x, \nabla g^*(\x'))  \right\}\subset \left\{ (\x, \x')\in \R^{2d}: \ \|\x- \x'\|\leq C\eps^{\frac{1}{d+2}} \right\} $$
        holds for every 
         $\eps\leq \eps_0 $. 
         \item For $\eps\leq \eps_0$, 
     \begin{multline*}
        \sup_{\x, \x'\in \Omega_0}\bigg\vert  \left(  C_\eps(\x)- D(\x, \nabla g^*(\x'))\right)_+
        - \left(  C_\eps(\x) -\frac{1}{2}\|\x-\x'\|_{\nabla^2 g^*(\x)}^2\right)_+\bigg\vert
        =o(\eps^{\frac{2}{d+2}}).
     \end{multline*}
     \end{enumerate}  
 \end{Lemma}

\subsection{Properties of the integrals}

The following results are well-known and easy to prove passing to spherical coordinates. We state them for further reference.

\begin{Proposition}\label{Lemma:elementary}
For any $a\geq0$,    it holds that 
    $$
    \int \left( a- \frac{\|\u\|^2}{2}\right)_+ \diff\ell_d(\u)= |a|^{\frac{d+2}{2}} C_d^{(1)}  ,$$
    where 
    $$
    C_d^{(1)} := \frac{2^{\frac{d+2}{2}}}{d(d+2)} \mathcal{H}^{d-1}(\mathcal{S}^{d-1}),
    $$
    and also
    $$
    \int \left(a- \frac{\|\u\|^2}{2}\right)_+^2 \diff\ell_d(\u)= |a|^{\frac{d+4}{2}} C_d^{(2)},
    $$
    where 
    $$
    C_d^{(2)}:= \frac{2^{\frac{d+6}{2}}}{d(d+2)(d+4)} \mathcal{H}^{d-1}(\mathcal{S}^{d-1}).
    $$
\end{Proposition}

\section{Lower  bound}\label{Section:lower}
The aim of this section is to show 
$$ \liminf_{\eps\to 0^+} \frac{\mathcal{T}_{2, \varepsilon, (\cdot)^2}(\rho_0, \rho_1) - \mathcal{W}_2^2(\rho_0, \rho_1)}{\eps^{\frac{2}{d+2}}}\geq {\frac{d^{\frac{d+4}{d+2}}(d+2)^{\frac{2}{d+2}}}{\left(\mathcal{H}^{d-1}(\mathcal{S}^{d-1}) \right)^{\frac{2}{d+2}}}}   \int \big(\rho_0(\x)\rho_1[\nabla g^*(\x)]\big)^{-\frac{1}{(d+2)}}\diff \rho_0(\x). $$
A sufficient condition for this to occur is given by the following result \textemdash see equation~\eqref{lowerMultby2} in the introduction, whose proof will be provided below. 
One of the key aspects of this proof is that we divide the set $\Omega_0$ into two, the points near its boundary or far away, namely
\begin{equation}\label{def:omega_delta}
 \Omega_0^{\delta} := \{ \x\in \Omega_0: \ {\rm dist  }(\x, \partial \Omega_0)> \delta \}\quad\text{ and }\quad\Omega_0^{\delta,c}:=\{ \x\in \Omega_0: \ {\rm dist  }(\x, \partial \Omega_0)\leq  \delta \}.
\end{equation}
These sets will also play a fundamental role in the proofs of Section~\ref{section:upper}, hence their importance. As the reader will see, being at a distance $\delta$ from the boundary will provide many useful uniform estimates that are in the core of our analysis. Then, making $\delta\to 0$ we will find that our results hold true.

\begin{Lemma}\label{Lemma:lowerbound}
    Let Assumption~\ref{Assumptions} hold.  Set $C_d:= 2^{\frac{d+2}{2}} \mathcal{H}^{d-1}(\mathcal{S}^{d-1}) \frac1{d(d+2)}$. Then 
    $$ \lim_{\eps\to 0^+} \frac{\Gamma_\eps(f_\eps, \frac{1}{2}\|\cdot\|^2 -g )- \frac{1}{2}\mathcal{W}_2^2(\rho_0, \rho_1)}{\eps^{\frac{2}{d+2}}}=\frac{d}{C_d^{\frac{2}{d+2}}}   \int  \big(\rho_0(\x)\rho_1[\nabla g^*(\x)]\big)^{-\frac{1}{(d+2)}}\diff \rho_0(\x),   $$
    where 
    $ {f}_\eps(\x)=  \frac{1}{2}\|\x\|^2 -g^*(\x)+ C_\eps(\x) $ for $ C_\eps(\x)={\eps^{\frac{2}{d+2}}}{C_d^{\frac{-2}{d+2}} \big(\rho_0(\x)\rho_1[\nabla g^*(\x)]\big)^{\frac{-1}{d+2}}} $.

\end{Lemma}
\begin{proof}
As 
 \begin{multline*}
  \Gamma_\eps(f_\eps, \frac{1}{2}\|\cdot\|^2 -g )=  \iint \bigg\{\frac{1}{2}\|\x\|^2 -g^*(\x)+ C_\eps(\x) + \frac{1}{2}\|\y\|^2 -g^*(\y)
   \\ -\frac1{2\eps}\left( C_\eps(\x) -g(\y)-g^*(\x)+\frac{1}{2}\|\x\|^2+\frac{1}{2}\|\y\|^2-\frac{1}{2}\|\x-\y\|^2\right)_+^2\bigg\} \diff\rho_0(\x) \diff \rho_1(\y)
\end{multline*}
and, by \eqref{dualrelationOT},
 $$ \mathcal{W}_2^2(\rho_0, \rho_1)=\int  \left(\frac{1}{2}\|\cdot\|^2-g^*\right)\diff \rho_0+\int  \left(\frac{1}{2}\|\cdot\|^2-g\right)\diff \rho_1,
$$
 the value of $$L=\lim_{\eps\to 0^+} \frac{\Gamma_\eps(f_\eps, \frac{1}{2}\|\cdot\|^2 -g )- \mathcal{W}_2^2(\rho_0, \rho_1)}{\eps^{\frac{2}{d+2}}}$$ is given by the limit (multiplied by $\eps^{-\frac{2}{d+2}}$) of
 
 \begin{align*}
     L_\eps&=\int C_\eps(\x) \diff \rho_0(\x) -\frac{1}{2}\underbrace{\frac{1}{ \eps} \iint \left(  C_\eps(\x)- D(\x, \nabla g^*(\x'))\right)_+^2 \diff \rho_0(\x)\diff \rho_0(\x')}_{L'_\eps}.
 \end{align*}
Since  
 $$ \eps^{-\frac{2}{d+2}}\int C_\eps(\x) \diff \rho_0(\x) = \int \frac{1}{C_d^{\frac{2}{d+2}} \big(\rho_0(\x)\rho_1[\nabla g^*(\x)]\big)^{\frac{1}{d+2}}} d\rho_0(\x), $$
 we only need to deal with the limit of $ \eps^{-\frac{2}{d+2}} L'_\eps$.
 
 Set now $\delta>0$ and define the sets $\Omega_0^\delta$ and $\Omega_0^{\delta,c}$ as in~\eqref{def:omega_delta}. Then, due to Lemma~\ref{LemmaDevelopmentLeps}, there exists $\eps_\delta$ such that for every $\eps\leq \eps_\delta$, the set $\Omega_0^{\delta, c}$ contains the set
$$ \mathbb{V}_\eps=\{ \x\in \Omega_0: \exists \, \x'\in \partial\Omega_0: \  C_\eps(\x)\geq  D(\x, \nabla g^*(\x')) \}. $$
 Therefore, we split 
$\eps^{-\frac{2}{d+2}} L'_\eps $ into two terms; 
\begin{equation}
    \label{eq:splitL1}
\mathbb{T}_1=\eps^{-\frac{d+4}{d+2}} \int_{\Omega_0^{\delta}} \int \left(  C_\eps(\x)- D(\x, \nabla g^*(\x'))\right)_+^2 \diff \rho_0(\x') \diff \rho_0(\x)
\end{equation}
and 
\begin{equation}
    \label{eq:splitL2}
   \mathbb{T}_2= \eps^{-\frac{d+4}{d+2}} 
 \int_{\Omega_0^{\delta,c}} \int \left(  C_\eps(\x)- D(\x, \nabla g^*(\x'))\right)_+^2 \diff \rho_0(\x')\diff \rho_0(\x),
\end{equation}
and estimate each one separately. 

{\it Estimate for \eqref{eq:splitL1}.}   Since  $\rho_0$ is uniformly continuous over $\Omega_0$ and the support of the function $\left(  C_\eps(\x)- D(\x, \nabla g^*(\x'))\right)_+$ concentrates uniformly on the diagonal $\x=\x'$ (see Lemma~\ref{LemmaDevelopmentLeps}), it holds that  
\begin{align*}
\mathbb{T}_1= \eps^{-\frac{d+4}{d+2}}\int_{\Omega_0^{\delta}} \int_{\Omega_0}  \left(  C_\eps(\x)- D(\x, \nabla g^*(\x'))\right)_+^2 \diff \ell_d(\x') (\rho_0(\x)+o\left(1 \right)) \diff \rho_0(\x').
\end{align*} 
Considering that $  \Omega_0^{\delta}$ and $\mathbb{V}_\eps$ are disjoint, it follows that the inner integral can be taken over $\R^d$, i.e., 
\begin{align*}
    \mathbb{T}_1= \eps^{-\frac{d+4}{d+2}}\int_{\Omega_0^{\delta}} \int_{\R^d}  \left(  C_\eps(\x)- D(\x, \nabla g^*(\x'))\right)_+^2 \diff \ell_d(\x') (\rho_0(\x)+o\left(1 \right)) \diff \rho_0(\x').
\end{align*} 
Lemma~\ref{LemmaDevelopmentLeps} implies 
\begin{align*}
    \mathbb{T}_1\asymp \eps^{-\frac{d+4}{d+2}}\int_{\Omega_0^{\delta}} \int_{\Omega_0}  \left(  C_\eps(\x)-\frac{1}{2}\langle \x-\x', \nabla^2 g^*(\x) (\x-\x') \rangle )\right)_+^2 \diff \ell_d(\x') {\rho}_0(\x) \diff \rho_0(\x).
\end{align*} 
We focus on the inner integral
$$I(\x)= {\rho}_0^2(\x) \int_{\Omega_0}  \left(  C_\eps(\x)-\frac{1}{2}\langle \x-\x', \nabla^2 g^*(\x) (\x-\x') \rangle )\right)_+^2 \diff \ell_d(\x') .$$ First we change variables $\u =  [\nabla^2 g^*(\x)]^{\frac{1}{2}} (\x-\x'),   $ to get 
$$
   I(\x)=  \frac{{\rho}_0^2(\x)}{\det(\nabla^2 g^*(\x))^{\frac{1}{2}} } \int \left(  C_\eps(\x)-\frac{1}{2}\|\u\|^2\right)_+^2 
    \diff \ell_d(\u).
$$
Using Proposition~\ref{Lemma:elementary}, we find
\begin{align*}
    I(\x)&=  \frac{\big(C_\eps(\x)\big)^{\frac{d+4}{2}} C_d^{(2)} {\rho}_0^2(\x)}{\det(\nabla^2 g^*(\x))^{\frac{1}{2}} } \\
    &=  \frac{ \eps^{\frac{d+4}{d+2}} C_d^{(2)} {\rho}_0^2(\x)}{C_d^{\frac{d+4}{d+2}} \big(\rho_0(\x)\rho_1[\nabla g^*(\x)]\big)^{\frac{d+4}{2(d+2)}} \det(\nabla^2 g^*(\x))^{\frac{1}{2}} }\\
    &= \frac{ \eps^{\frac{d+4}{d+2}} C_d^{(2)}  \big(\rho_0(\x) \rho_1[\nabla g^*(\x)]\big)^{\frac{1}{2}}}{C_d^{\frac{d+4}{d+2}} \big(\rho_0(\x)\rho_1[\nabla g^*(\x)]\big)^{\frac{d+4}{2(d+2)}}  }\rho_0(\x)\\
    &= \frac{ \eps^{\frac{d+4}{d+2}} C_d^{(2)}  }{C_d^{\frac{d+4}{d+2}} \big(\rho_0(\x)\rho_1[\nabla g^*(\x)]\big)^{\frac{1}{(d+2)}}  }\rho_0(\x), 
\end{align*}
    where 
    $ C_d^{(2)}:= \frac{2^{\frac{d+6}{2}}}{d(d+2)(d+4)} \mathcal{H}^{d-1}(\mathcal{S}^{d-1}).$
Hence, 
$$  \mathbb{T}_1 \asymp \frac{  C_d^{(2)}  }{C_d^{\frac{d+4}{d+2}}   }\int_{\Omega_0^{\delta}} \big(\rho_0(\x)\rho_1[\nabla g^*(\x)]\big)^{-\frac{1}{(d+2)}}\diff \rho_0(\x). $$
We simplify the constant $C_d^{(2)} C_d^{-\frac{d+4}{d+2}}$ as follows:
\begin{align*}
C_d^{(2)} C_d^{-1 -\frac{2}{d+2}} 
&= C_d^{ -\frac{2}{d+2}  } \frac{2^{\frac{d+6}{2}}}{d(d+2)(d+4)} \mathcal{H}^{d-1}(\mathcal{S}^{d-1}) 2^{-\frac{d+2}{2}} \big(\mathcal{H}^{d-1}(\mathcal{S}^{d-1})\big)^{-1} d(d+2) \\
&= \frac{4}{d+4}C_d^{ -\frac{2}{d+2}  }, 
\end{align*}
which leads to
\begin{equation}
    \label{T1} 
   \mathbb{T}_1 \asymp \frac{4}{d+4}C_d^{ -\frac{2}{d+2}  }\int_{\Omega_0^{\delta}} \big(\rho_0(\x)\rho_1[\nabla g^*(\x)]\big)^{-\frac{1}{(d+2)}}\diff \rho_0(\x). 
\end{equation}
{\it Estimate for \eqref{eq:splitL2}.} First note that under Assumption~\ref{Assumptions}, 
\begin{equation}
    \label{eq:boundCeps}
    C_\eps(\x)\leq \frac{\eps^{\frac{2}{d+2}}}{ C_d^{\frac{2}{d+2}}\lambda^{\frac{2}{d+2}}} .
\end{equation}
Hence, due to Lemma~\ref{LemmaDiver} and the fourth point of Assumption~\ref{Assumptions}, we have the bound
\begin{align*}
0\leq&  \mathbb{T}_2\leq \eps^{-\frac{d+4}{d+2}}\int_{\Omega_0^{\delta, c}} \int \left( \frac{\eps^{\frac{2}{d+2}}}{C_d^{\frac{2}{d+2}}\lambda^{\frac{2}{d+2}}} -c\|\x-\x'\|^2\right)_+^2 \diff \rho_0(\x) \diff \rho_0(\x')\\
&\leq  \Lambda\,  \eps^{-\frac{d+4}{d+2}}\int_{\Omega_0^{\delta, c}} \int \left( \frac{\eps^{\frac{2}{d+2}}}{C_d^{\frac{2}{d+2}}\lambda^{\frac{2}{d+2}}} -c\|\x-\x'\|^2\right)_+^2 \diff \ell_d(\x)  \diff \rho_0(\x')\\
&\leq  \frac{\Lambda}{\lambda^{\frac{2}{d+2}}}   \eps^{-\frac{d}{d+2}}\int_{\Omega_0^{\delta, c}}\ell_d\left(\left\{\u : c\|\u\|^2\leq \frac{\eps^{\frac{2}{d+2}}}{C_d^{\frac{2}{d+2}}\lambda^{\frac{2}{d+2}}}   \right\}\right)  \diff \rho_0(\x'),
\end{align*}
for some positive constant $c$. As a consequence, there exists a constant $c'$ such that 
$$ 0\leq  \mathbb{T}_2\leq c' \rho_0(\Omega_0^{\delta, c}). $$
We can now conclude the proof by noticing that
\begin{align*}
    \liminf_{\eps\to 0^+}\ &\frac{\Gamma_\eps(f_\eps, g)- \mathcal{W}_2^2(\rho_0, \rho_1)}{\eps^{\frac{2}{d+2}}} \\
    &\geq \frac{d}{C_d^{\frac{2}{d+2}}} \int_{\Omega_0^{\delta}} \big(\rho_0(\x)\rho_1[\nabla g^*(\x)]\big)^{-\frac{1}{(d+2)}}\diff \rho_0(\x) -C\rho_0(\Omega_0^{\delta, c}) 
\end{align*}
and 
\begin{align*}
    \limsup_{\eps\to 0^+} &\frac{\Gamma_\eps(f_\eps, g)- \mathcal{W}_2^2(\rho_0, \rho_1)}{\eps^{\frac{2}{d+2}}} \\
    &\leq \frac{d}{C_d^{\frac{2}{d+2}}}   \int_{\Omega_0^{\delta}} \big(\rho_0(\x)\rho_1[\nabla g^*(\x)]\big)^{-\frac{1}{(d+2)}}\diff \rho_0(\x) +C\rho_0(\Omega_0^{\delta, c}) 
\end{align*}
hold for any $\delta>0$. Letting $\delta\to 0$ we obtain the result.
\end{proof}

\section{Upper bound}\label{section:upper}

In this section, we establish an upper bound that is asymptotically equivalent to the lower bound.
\begin{figure}[h!]
  \centering
    \includegraphics[width=6 cm,height=4cm, trim={10cm 4cm 10cm 4cm},clip ]{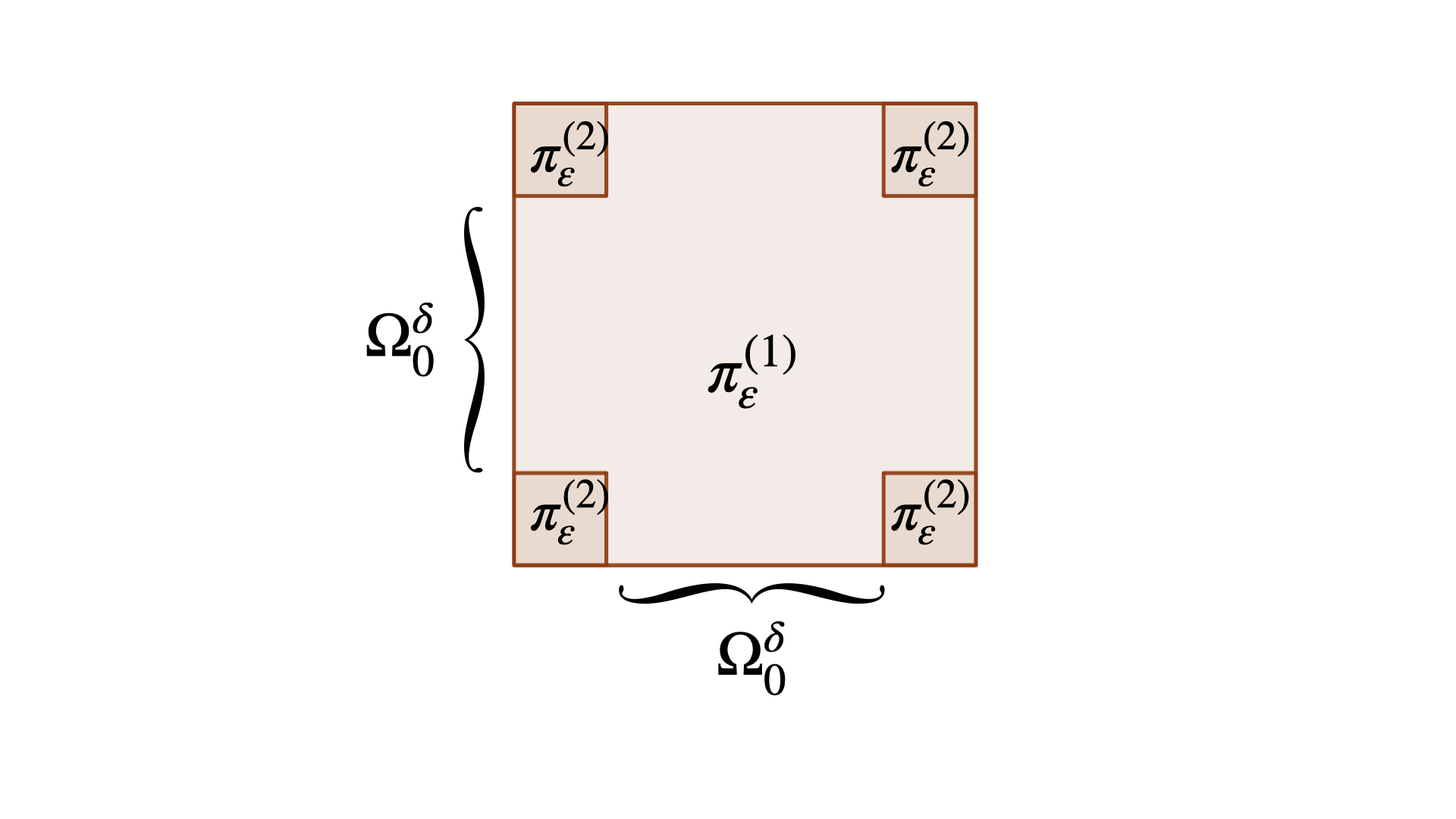} 
    \includegraphics[width=6 cm,height=4cm, trim={10cm 4cm 10cm 4cm},clip ]{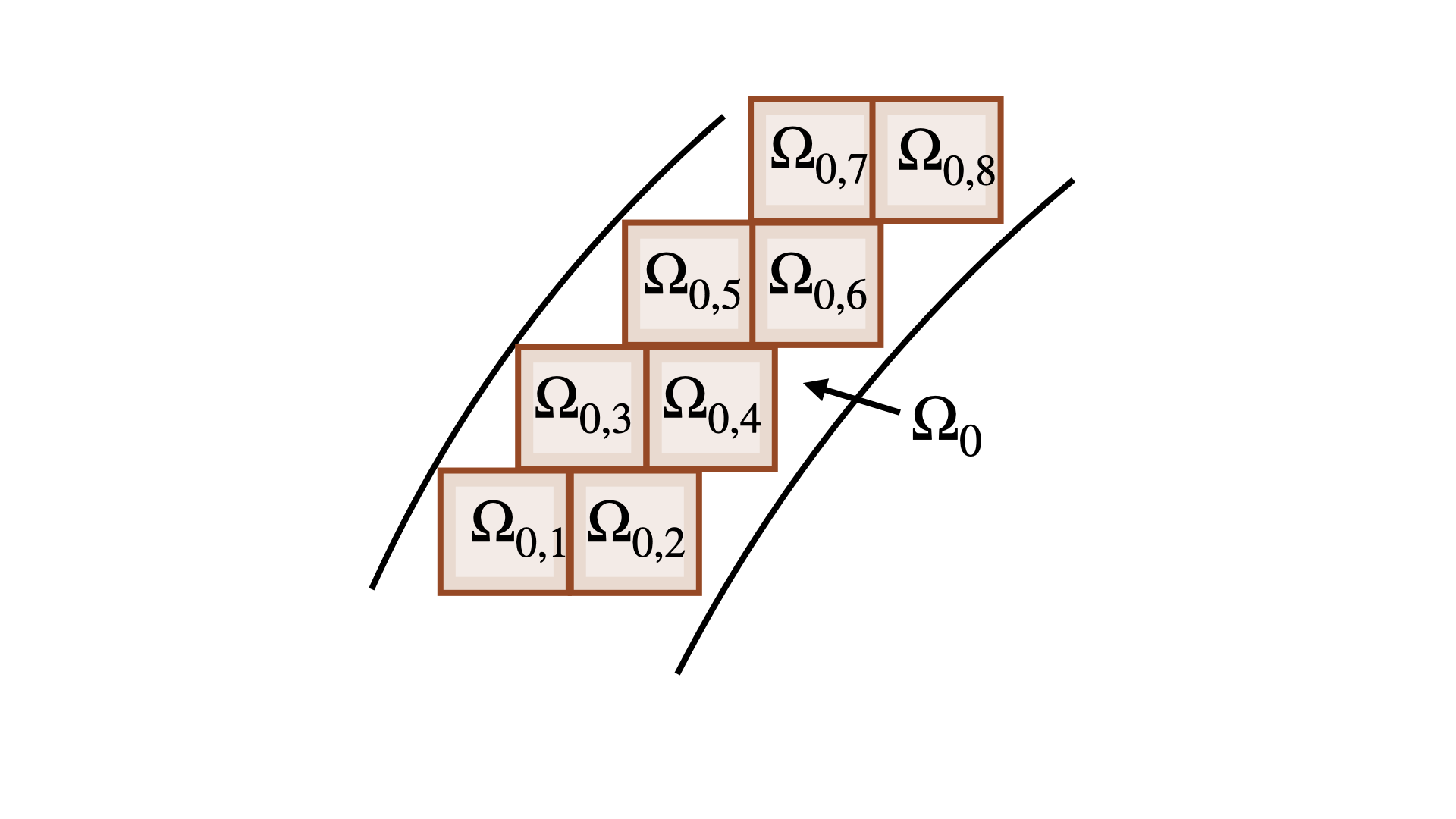}
  \caption{Proof Strategy. First, we prove the result assuming that \( \Omega_0 \) is convex (left image). We will define two different couplings, $\pi^{(1)}_\eps$ in the points far from the boundary of $\Omega_0$ and $\pi^{(2)}_\eps$ for the points close to said boundary, creating some sort of a frame for \( \Omega_0^\delta \). We will see that with \(\delta\) (the thickness of the frame) fixed, one can obtain limits as \(\varepsilon \to 0\) that differ from the lower bounds by a function of \(\delta\) which tends to 0 as \(\delta \to 0\). When the domain is not convex (right image), we will fix a tessellation of squares in a set contained within \( \Omega_0 \), where we will repeat the previous procedure in each square and create a new coupling in the part not covered by the union of the squares.
}
  \label{fig:images}
\end{figure}
The proof is quite technical. We will proceed as shown in Figure~\ref{fig:images}. First, we will assume that the set \( \Omega_0 \) is convex and we will divide it between its interior $\Omega_0^\delta$ and the points close to the boundary, $\Omega_0^{\delta,c}$. One can imagine $\Omega_0$ as being a \textit{window} where $\Omega_0^\delta$ is its \textit{glass} and $\Omega_0^{\delta,c}$ its \textit{frame} of width \(\delta\). In the glass part, we will define a coupling that achieves the appropriate limit, while in the frame, we will create one that tends to $0$ as we let $\eps\to 0$ first and then  \(\delta \to 0\). For the general case, we will cover most of $\Omega_0$ by composing a sort of \textit{stained-glass window} formed by convex windows $\Omega_{0,1}, \Omega_{0,2}, etc...$. In each of the windows contained within the interior of \( \Omega_0 \), we will use the same coupling as before. In those touching the boundary of \( \Omega_0 \), we will use a coupling that does not blow up and decreases as the number of windows increases, albeit using many of smaller size.

Now we will present some technical results for bounding frame couplings in both the convex and general $\Omega_0$ cases.


\begin{Lemma}\label{Lemma:firstEstimates}
Let $\rho_0$ and $g$ be as in Assumption~\ref{Assumptions}. Let $s:\Omega_0\mapsto [a,b]$ for some $0<a<b<\infty
$ be a given function and $\Omega\subset \Omega_0$ be any open set such that 
    there exists $r_0>0$ and $\theta\in (0,\pi)$ such that for  every $i\in \{1, 2\}$ and $\x\in \Omega$ there exists a cone
    \begin{equation}
        \label{ConeCond}
        \mathcal{C}_{\x, {\bf v},r_0, \sigma}=\left\{ \x+{\bf u}: \|{\bf v}\|\|{\bf u}\| \cos\left(\frac{1}{2}\theta\right) \leq \langle {\bf v} ,  {\bf u }
 \rangle \leq \|{\bf v}\|r_0  \right\}
    \end{equation} 
    with vertex $\x$,  height $r_0$ and angle $\theta$ such that $\mathcal{C}_{\x, {\bf v},r_0, \sigma}\subset \Omega$.
Then there exists  $ \eps_0$ and $C=C(c,a,b)$
such that for $\eps\leq \eps_0$ the following holds:
\begin{enumerate}
    \item \label{Lemma:firstEstimates1} There exists a function 
$ \psi_\eps $ such that
$$s(\x)=\frac{1}{\eps}\int_{\Omega} (\psi_\eps (\x)-D(\x, \nabla g^*(\x') )_+\diff \rho_0(\x') , \quad \text{for all } \x\in \Omega.   $$
\item \label{Lemma:firstEstimates2} There exists $\eps_0$ and $C>0$ such that for every  $\x\in \Omega$ and $\eps\leq \eps_0$, 
\begin{equation}
    \label{eq:maxceps}
    \frac{\varepsilon^{\frac{2}{d+2}} }{C} \leq  \psi_\eps(\x) \leq \varepsilon^{\frac{2}{d+2}} C,   
\end{equation}
and 
\begin{equation}
    \label{sizesupport}
    \frac{\varepsilon^{\frac{d}{d+2}} }{C}\leq \ell_d(\mathcal{S}_{\varepsilon, \x})\leq C\varepsilon^{\frac{d}{d+2}},
\end{equation}
where $\mathcal{S}_{\varepsilon, \x}= \{ \x'\in \Omega_0:  \psi_\eps (\x)-D(\x, \nabla g^*(\x')  \geq 0\}.$
\item \label{Lemma:firstEstimates3} For every  $\x'\in \Omega$, there exists a $c'>0$ such that
$$ \frac{1}{\eps}\int_{\Omega} (\psi_\eps (\x)-D(\x, \nabla g^*(\x') )_+\diff \rho_0(\x) \geq c'. $$
\end{enumerate} 

\end{Lemma}
\begin{proof} We prove each point separately. 

{\it Proof of \ref{Lemma:firstEstimates1}.}
Set $\x\in \Omega_0$ and 
$$ \Theta(a)=\frac{1}{2\,\eps}\int_{\Omega} (a-D(\x, \nabla g^*(\x') )_+^2\diff \rho_0(\x')-a s(\x).  $$
Since $s(\x)>0$, we have $ \lim_{a\to \pm \infty} \Theta(a)=+\infty$. Moreover, as 
$\Theta$ is   convex and differentiable,  there exists a minimizer $a(\x)$ which satisfies 
$$ \frac{1}{\,\eps}\int_{\Omega} (a(\x)-D(\x, \nabla g^*(\x') )_+\diff \rho_0(\x')= s(\x) . $$
This proves the existence. 

{\it Proof of \ref{Lemma:firstEstimates2}.}Again, since $s(\x)>0$ for all $\x\in \Omega_0$,  $\psi_\eps$  is non-negative  and Lipschitz with constant $2\cdot{\rm diam}(\Omega)$  (see the proof of \cite[Lemma~2.5]{nutz2024quadratically}). This implies that 
$ \psi_\eps$   tends uniformly to $0$ in $\Omega$ as $\eps \to 0$. 
Lemma~\ref{LemmaDiver} implies that
\begin{align*}
    \eps s(\x)
    &\leq \int_{\Omega} \left(\psi_\eps(\x)- \frac{\sigma_m(g)}{2}\|\x-\x'\|^2\right)_+ \diff \rho_0(\x')\\
    &\leq \Lambda \int_{\Omega} \left(\psi_\eps(\x)- \frac{\sigma_m(g)}{2}\|\x-\x'\|^2\right)_+ \diff \ell_d(\x')\\
     &\leq \Lambda \int_{\R^d} \left(\psi_\eps(\x)- \frac{\sigma_m(g)}{2}\|\x-\x'\|^2\right)_+ \diff \ell_d(\x')\\
    &\leq C  (\psi_\eps(\x))^{\frac{d+2}{2}},
\end{align*}
for a certain constant $C>0$. Hence, the lower bound in \eqref{eq:maxceps} follows.  Next, 
\begin{align*}
    \eps s(\x) &\geq \int_{\Omega} (\psi_\eps(\x)- \frac{\sigma_M(g)}{2}\|\x-\x'\|^2)_+ \diff \rho_0(\x')\\
    &\geq \lambda \int_{\Omega} \left(\psi_\eps(\x)- \frac{\sigma_M(g)}{2} \|\x-\x'\|^2\right)_+\diff \ell_d(\x').
\end{align*}
By assumption, $\Omega$ contains the cone $ \mathcal{C}_{\x, {\bf v},r_0, \theta}$. As a consequence, 
\begin{align*}
    \eps s(\x) & \geq \lambda\int_{\mathcal{C}_{\x, {\bf v},r_0, \theta}} \left(\psi_\eps(\x)- \frac{\sigma_M(g)}{2} \|\x-\x'\|^2\right)_+\diff \ell_d(\x')\\
    &=\lambda \int_{\mathcal{C}_{\x, {\bf v},r_0, \theta}\cap \mathbb{B}\left(\x,  \left(\frac{2\psi_\eps(\x)}{\sigma_M(g)}\right)^{\frac{1}{2}}\right)} \left\{\psi_\eps(\x)- \frac{\sigma_M(g)}{2} \|\x-\x'\|^2\right\}\diff \ell_d(\x').
\end{align*}
 Since  $ \psi_\eps$   tends uniformly to $0$ in $\Omega$, we can find $\eps_0>0$ such that $$\|\psi_\eps\|_{\mathcal{C}(\Omega)}\leq \left(\frac{2\,r_0}{\sigma_M(g)}\right)^2$$ for all $\eps\leq \eps_0$. This yields 
 $$  \eps s(\x) \geq \lambda \int_{\mathcal{C}_{\x, {\bf v},\psi_\eps(\x), \theta}\cap \mathbb{B}\left(\x,  \left(\frac{2\psi_\eps(\x)}{\sigma_M(g)}\right)^{\frac{1}{2}}\right)} \left\{\psi_\eps(\x)- \frac{\sigma_M(g)}{2} \|\x-\x'\|^2\right\}\diff \ell_d(\x').$$
Hence, there exists a constant $C(\theta,d, \sigma_M(g))>0$ depending only on $d$, $\sigma_M(g)$ and the $\mathcal{H}^{d-1}$-measure of the sector 
$$ \mathcal{S}^{d-1} \cap \left\{ {\bf u}: \|{\bf e}_1\|\|{\bf u}\| \cos\left(\frac{1}{2}\theta\right) \leq \langle {\bf e}_1 ,  {\bf u }
 \rangle   \right\},$$
where $ {\bf e}_1$ is the first element of the canonical basis of $\R^d$,  such that 
$$  \eps s(\x)  \geq C(d, \theta) \lambda \left( \psi_\eps(\x)\right)^{\frac{d+2}{2}}. $$
This gives the upper bound in \eqref{eq:maxceps}.  From here it is easy to obtain 
\eqref{sizesupport} via the nesting
$$ \mathbb{B}\left(\x, \frac{1}{C}\psi_\eps(\x)\right)\subset \mathcal{S}_{\varepsilon, \x}\subset  \mathbb{B}\left(\x, C\, \psi_\eps(\x)\right),  $$
which are consequence of Lemma~\ref{LemmaDiver}.

{\it Proof of \ref{Lemma:firstEstimates3}.} 
By the previous points, There exist $ \eps_0$ and $c> 0$
such that for $\eps\leq \eps_0$ 
\begin{align*}
    \int_{\Omega} (\psi_\eps (\x)-D(\x, \nabla g^*(\x') )_+ d\rho_0(\x)&\geq \int_{\Omega} \left(c\eps^{\frac{2}{d+2}}- D(\x, \nabla g^*(\x'))\right)_+ d\rho_0(\x) \\
    &\geq \int_{\Omega} \left(\frac{c}{C}\psi_\eps (\x')- D(\x, \nabla g^*(\x'))\right)_+ d\rho_0(\x)\\
 \text{(by Lemma~\ref{LemmaDiver}) }  & \geq c_1 \int_{\Omega} \left(\psi_\eps (\x')- D(\x', \nabla g^*(\x))\right)_+ d\rho_0(\x)\\
    &\geq c_1 \,\eps\, \inf_{\x}s(\x).
\end{align*}
Defining $c':=c_1 \, \inf_{\x}s(\x)$ we obtain the desired result.
\end{proof}

\subsection{Feasible coupling for the interior $\Omega_0^\delta$ of a convex $\Omega_0$}\label{section:convex}

In this section we will define a possible coupling for the set $\Omega_0^\delta$. We will do so, in fact, by finding a coupling  $\mu_\eps\in \Pi(\rho_0, \rho_0) $ that attains the correct limit when restricted  to $(\Omega_0^\delta\times \Omega_0)\cup (\Omega_0\times \Omega_0^\delta) $ for $\delta>0$. Then, we will restrict this coupling only to $(\Omega_0^\delta\times \Omega_0)\cup (\Omega_0\times \Omega_0^\delta)$, but by doing so we will lose the part of the mass corresponding with $\Omega_0^\delta\times \Omega_0^\delta $, which we will compensate with an extra term; see Section~\ref{sec:fixing_the_coupling} later on. 

We define, for $\x,\x'\in \Omega_0\times\Omega_0$,
\[ m_\varepsilon ({\bf x}, {\bf x}')=\frac{1}{2\, \eps}\left(C_\eps(\x)+C_\eps(\x')-\frac{1}{2}\|\x-\x'\|_{\nabla^2 g^*(\x) }^2- \frac{1}{2}\|\x-\x'\|_{\nabla^2 g^*(\x') }^2 \right)_+ . \]
The following result is a direct consequence of Lemma~\ref{LemmaDiver}, the definition of $C_\eps$, the uniform continuity of $\rho_0$, $\rho_1$, and the transport map $\nabla g$, combined with the strict positivity of both $\rho_0$ and $\rho_1$. 
\begin{Lemma}\label{lemma:symetric}
  Let Assumption~\ref{Assumptions} hold. Then there exists $C$ such that 
  \begin{equation}
      \label{eq:concentratedSuppoet}
      \|\x-\x'\|\leq C \eps^{\frac{2}{d+2}} ,
  \end{equation}
  for all $ (\x, \x')\in {\rm supp}(m_\eps) $
and 
  \[\eps\sup_{\x, \x'\in \Omega_0}\left\vert  m_\varepsilon ({\bf x}, {\bf x}')-\frac{1}{ \eps}\left(C_\eps(\x)- \frac{1}{2}\|\x-\x'\|_{\nabla^2 g^*(\x)}^2 \right)_+\right\vert =o(\eps^{\frac{2}{d+2}}).  \]
\end{Lemma}
We define the 
function 
$$ {\xi}_{\eps} (\x, \x)=\frac{m_\eps(\x , \x')}{\int m_\eps(\x , \x') \diff (\rho_0\otimes \rho_0)(\x, \x')},  $$
which is a density (it is positive and integrates one) with respect to $\rho_0\otimes \rho_0$. We define also 
$$ \rho_\eps(\x)=\int {\xi}_{\eps} (\x, \x') \diff \rho_0(\x')$$
as the marginal density (w.r.t. $\rho_0$), which by symmetry agrees with 
$\rho_\eps(\x')=\int {\xi}_{\eps} (\x, \x') \diff \rho_0(\x).$
Moreover it satisfies the following properties. We recall the definition of $\Omega_0^\delta$ in~\eqref{def:omega_delta}.

\begin{Lemma}\label{LemmaBoundsUpper} Let Assumption~\ref{Assumptions} hold. Then there exist two positive constants $c \leq C\in (0, \infty)$ and a $delta_0>0$ such that for every $\delta<\delta_0, $ there exists an $\eps_\delta$ such that:
\begin{enumerate}
\item \label{LemmaBoundsUpperPoint1} $\rho_\eps\, \diff \rho_0 \in \mathcal{P}(\Omega_0)$ and  $ \xi_{\eps} \, \diff (\rho_0 \otimes \rho_0)\in \Pi(\rho_\eps \diff \rho_0, \rho_\eps \diff \rho_0)$, 
\item \label{LemmaBoundsUpperPoint2}  
    $ c \leq  \rho_\eps  \leq  C $ in  $\Omega_0$,
\item \label{LemmaBoundsUpperPoint3}  $\rho_\eps\in \mathcal{C}^{0, \alpha}\left(\overline{\Omega_0^{\delta}}\right)$ with 
$ \|\rho_\eps\|_{\mathcal{C}^{0, \alpha}\left(\overline{\Omega_0^{\delta}}\right)} \leq C' \quad \text{for some positive constant }C', and $
\item\label{LemmaBoundsUpperPoint4} as $\eps\to 0$, 
    $$\|\rho_\eps -1\|_{\mathcal{C}\left(\overline{\Omega_0^{\delta}}\right)}\to 0 \quad {\rm and}\quad  \left\|\int m_{\eps} (\cdot, \x') \diff \rho_0(\x') -1\right\|_{\mathcal{C}\left(\overline{\Omega_0^{\delta}}\right)}\to 0 . $$
\end{enumerate}
\end{Lemma}
\begin{proof} Point \ref{LemmaBoundsUpperPoint1} holds by construction. We prove each of the remaining points separately. 

{\it Proof of \ref{LemmaBoundsUpperPoint2}} By definition of ${\xi}_{\eps}$, it is enough to show the existence of two positive constants, $c$ and $C$ such that 
$ c \leq \int m_\eps(\x, \x') \diff \rho_0(\x) \leq C $
for all $\x\in \Omega_0$. 
Set $\x, \x'\in \Omega_0$.
 By Assumption~\ref{Assumptions}, we have 
 \begin{equation}
     \label{boundCandDiv}
  C_\eps(\x)\leq \frac{ \eps^{\frac{2}{d+2}}}{C_d^{\frac{2}{d+2}}\lambda^{\frac{2}{d+2}}} \quad  {\rm and} \quad \frac{1}{2}\|\x-\x'\|_{\nabla^2 {g^*}(\x) }^2 \geq \frac{\sigma_m(g)}{2}\|\x-\x'\|^{{2}} .
 \end{equation}
As a consequence, 
\begin{align*}
    \int m_\eps(\x, \x') \diff \rho_0(\x)&\leq  \int \frac{1}{ \eps}\left( \frac{\eps^{\frac{2}{d+2}}}{C_d^{\frac{2}{d+2}} \lambda^{\frac{2}{d+2}}} -\frac{\sigma_m(g)}{2}\|\x-\x'\|^2 \right)_+ \diff \rho_0(\x)\\
    &\leq \Lambda\int \frac{1}{ \eps}\left( \frac{ \eps^{\frac{2}{d+2}}}{C_d^{\frac{2}{d+2}}\lambda^{\frac{2}{d+2}}} -\frac{\sigma_m(g)}{2}\|\x-\x'\|^2 \right)_+ \diff \ell_d(\x)\\
    &=\Lambda\int \frac{1}{ \eps}\left( \frac{C_d \eps^{\frac{2}{d+2}}}{C_d^{\frac{2}{d+2}}\lambda^{\frac{2}{d+2}}} -\frac{\sigma_m(g)}{2}\|\x\|^2 \right)_+ \diff \ell_d(\x).
\end{align*}
The latter quantity, call it $C$, does not depend on $\eps$ and is finite, see Proposition~\ref{Lemma:elementary}.   The reverse inequality 
$$ \int m_\eps(\x, \x') \diff \rho_0(\x) \geq \lambda\int_{\Omega_0} \frac{1}{ \eps}\left( \frac{ \eps^{\frac{2}{d+2}}}{C_d^{\frac{2}{d+2}}\Lambda^{\frac{2}{d+2}}} -\frac{\sigma_M(g)}{2}\|\x-\x'\|^2 \right)_+ \diff \ell_d(\x) $$
holds by the same means. From here, we conclude as in the proof  Lemma~\ref{Lemma:firstEstimates}, claim~\ref{LemmaBoundsUpperPoint2}.

{\it Proof of \ref{LemmaBoundsUpperPoint3}.} Since  $\int m_\eps(\x , \x') \diff (\rho_0\otimes \rho_0)(\x, \x')$ and $\rho_0$ are uniformly bounded away from $0$ and $\infty$, it is enough to upper bound the $\mathcal{C}^{0, \alpha}\left(\overline{\Omega_0^{\delta}}\right)$-norm of  the function 
$$ \z\mapsto \int m_\eps(\z , \x) \diff \rho_0(\x). $$  
Set 
$ \z, \z'\in \Omega_0^\delta $ and assume w.l.o.g. that 
$\|\z-\z'\|^\alpha\leq  \frac{\sigma_m(g)}{4C_1}, $
 where  $C_1>0$ is the constant  such that 
\begin{equation}
    \label{boundHolderg}
    \vert \|\x\|_{\nabla^2 g^*({\bf a}) }^2- \|\x\|_{\nabla^2 g^*({\bf b}) }^2 \vert \leq C_1 \, \|\x\|^2 \|{\bf a}-{\bf b}\|^{\alpha},\quad \text{for all } {\bf a}, {\bf b}, \ \x\in \Omega_0
\end{equation}
which due to the fact that $\nabla^2 g(\z') $ is H\"older  continuous.
Since $ C_d^{\frac{-2}{d+2}}\ \big( \rho_0(\z)\rho_1(\nabla g^*(\z)) \big)^{\frac{-1}{d+2}}$  is also a H\"older  continuous function, here exists $C_2>0$ 
\begin{equation}
    \label{ModulusofCeps}
    |C_\eps(\z')-C_\eps(\z)| \leq C_2\|\z-\z'\|^{\alpha} \eps^{\frac{2}{d+2}}. 
\end{equation}
We aim at showing that 
\begin{equation}\label{claimHolder}
    \int m_\eps(\z , \x) \diff \rho_0(\x)\leq  \int m_\eps(\z' , \x) \diff \rho_0(\x) + C \| \z- \z'\|^\alpha, 
\end{equation}
for some constant $C>0$. We call 
$$ \mathcal{S}_\z := \left\{\x\in \Omega_0:  C_\eps(\x)+C_\eps({\bf z}) -\frac{1}{2}\|\x-\z\|_{\nabla^2 {g^*}(\x) }^2- \frac{1}{2}\|\x-\z\|_{\nabla^2 {g^*}(\z) }^2>0\right\}$$
and observe that, due to \eqref{eq:concentratedSuppoet}, there exists $C_3>0$ such that  $  \sup_{{\bf u}\in \mathcal{S}_\z } \|{\bf u}-{\bf z}\|\leq C_3 \eps^{\frac{1}{d+2}}.$
Hence, we can find $\eps_\delta>0$ such that  $\bigcup_{{\bf u}\in \Omega^\delta} \mathcal{S}_{{\bf u}}\subset \Omega^{\frac{\delta}{2}}  $ for all $\eps\leq \eps_\delta$. The rest of the proof always assume that $\eps\leq \eps_\delta$.  Hence,
\begin{align*}
    &\int m_\eps(\z , \x) \diff \rho_0(\x)\\
    &= \frac{1}{2\, \eps}\int_{\Omega_0} \bigg(C_\eps(\x)+C_\eps({\bf z}) -\frac{1}{2}\|\x-\z\|_{\nabla^2 {g^*}(\x) }^2- \frac{1}{2}\|\x-\z\|_{\nabla^2 {g^*}(\z) }^2 \bigg)_+  \rho_0(\x)\diff\ell_d(\x)\\
    & = \frac{1}{2\, \eps}\int_{\R^d} \bigg(C_\eps(\x)+C_\eps({\bf z}) -\frac{1}{2}\|\x-\z\|_{\nabla^2 {g^*}(\x) }^2- \frac{1}{2}\|\x-\z\|_{\nabla^2 {g^*}(\z) }^2 \bigg)_+  \rho_0(\x)\diff\ell_d(\x).  
\end{align*}
 Making the change of variables $\x'  = T_{\z,\z'}(\x)= \x +\z'- \z$, 
one gets 
\begin{align*}
    &\int m_\eps(\z , \x) \diff \rho_0(\x)
    \\
    & = \frac{1}{2\, \eps}\int_{\R^d} \bigg(C_\eps(\x' + \z - \z' )+C_\eps({\bf z}) -\frac{1}{2}\|\x'-\z'\|_{\nabla^2 {g^*}(\x' + \z - \z' ) }^2- \frac{1}{2}\|\x'-\z'\|_{\nabla^2 {g^*}(\z) }^2 \bigg)_+ \\
    & \qquad \qquad \rho_0(\x'-\z'+\z) \diff\ell_d(\x'). 
\end{align*}
Via \eqref{boundHolderg} and \eqref{ModulusofCeps}, the bound
\begin{align*}
   &\int m_\eps(\z , \x) \diff \rho_0(\x)\\
   &\leq  \frac{1}{2\, \eps}\int_{\R^d} \bigg(C_\eps(\x')+C_\eps({\bf z}') -\frac{1}{2}\|\x'-\z'\|_{\nabla^2 {g^*}(\x' ) }^2-\frac{1}{2}\|\x'-\z'\|_{\nabla^2 {g^*}(\z') }^2\\
   & \qquad\qquad + 2\, C_2\eps^{\frac{2}{d+2}}\|\z-\z'\|^\alpha+ 2\,C_1\|\x'-\z'\|^2\|\z-\z'\|^\alpha \bigg)_+
\rho_0(\x'-\z'+\z) \diff\ell_d(\x') 
\end{align*}
holds. 
Let $\mathcal{S}'$ be the set where the quantity inside the integral is positive and 
$$ C_5=\left(\frac{4 C_1}{\sigma_m(g)}\left(  \frac{C'\sigma_m(g)}{4C_1}+ \frac{1}{C_d^{\frac{1}{d+2}}\lambda^{\frac{2}{d+2}}}\right)\right)^{\frac{1}{2}}.$$ Since $\|\z-\z'\|^\alpha\leq  \frac{\sigma_m(g)}{4C_1}, $ the nesting 
$ \mathcal{S}' \subset \mathbb{B}\left({\bf z}',  \eps^{\frac{1}{d+2}}C_5\right)$
is derived.  This combined with the bound on the regularity of the density
$\rho_0(\x'-\z'+\z)\leq \rho_0(\x')+C_4\|\z-\z'\|^\alpha $
yield 
\begin{align*}
    &\int m_\eps(\z , \x) \diff \rho_0(\x) \\
    &\leq  \frac{1}{2\, \eps}\int_{\R^d} \bigg(C_\eps(\x')+C_\eps({\bf z}') -\frac{1}{2}\|\x'-\z'\|_{\nabla^2 {g^*}(\x' ) }^2-\frac{1}{2}\|\x'-\z'\|_{\nabla^2 {g^*}(\z') }^2\bigg)_+ \rho_0(\x') \diff\ell_d(\x')\\
    &+C_6 \|\z-\z'\|^{\alpha},
\end{align*}
where 
$$ C_6 =2C_2C_5^{d}\ell_{d}(\mathbb{B}(0,1)) \Lambda +2C_1 C_5^{d+2} \Lambda +C_5 \sup_{\x'\in \Omega^\delta} \int m_\eps(\x, \x') \diff \ell_d(\x), $$
 which is finite (see the proof of  \ref{LemmaBoundsUpperPoint2}).

This yields the first part of the H\"older inequality statement. 
The reverse inequality holds due to symmetry,  exchanging $\z$ by $\z'$.

{\it Proof of \ref{LemmaBoundsUpperPoint4}} 
By \eqref{ModulusofCeps}, Lemma~\ref{LemmaDiver} and the fact that the support of $m_\varepsilon$ decreases with rate $\eps^{\frac{2}{d+2}}$, we have 
\[ m_\varepsilon ({\bf x}, {\bf x}')=\frac{1}{ \eps}\left( C_\eps(\x) -\frac{1}{2}\|\x-\x'\|_{\nabla^2 g^*(\x) }^2 +o(\eps^{\frac{2}{d+2}}) \right)_+ . \]

Set now $\delta>0$ and $\x\in \Omega_0^\delta$. Then there exists $\eps_0>0$ small enough such that 
\[ \int_{\Omega_0^\delta} \left( C_\eps(\x)-\frac{1}{2}\|\x-\x'\|_{\nabla^2 g^*(\x) }^2\right)_+ \diff\ell_d(\x')
    = \int  \left( C_\eps(\x)-\frac{1}{2}\|\x-\x'\|_{\nabla^2 g^*(\x) }^2 \right)_+ \diff\ell_d(\x'),\]
for all $\eps\leq \eps_0$. Now we finish the argument. The change of variables 
$ \u =  [D^2g^* (\x)]^{\frac{1}{2}} (\x-\x') $
 yields
\begin{align*}
    \int m_\eps(\x, \x') d\rho(\x')&= \frac{\rho_0(\x) (1+o(1))}{\det(D^2g^* (\x))^{\frac{1}{2}}} \int_{\Omega_0^\delta} \left( C_\eps(\x)-\frac{\|\u\|^2}{2}\right)_+ d\ell_d(\u)\\
    &= \frac{\rho_0(\x) (1+o(1))}{\det(D^2g^* (\x))^{\frac{1}{2}} (\rho_0(\x) \rho_1[\nabla g^*(\x)]))^{\frac{1}{2}}}.
\end{align*}
The relation $\det(D^2g^* (\x))= \frac{\rho_0(\x) }{\rho_1[\nabla g^*(\x)]}$ yields
\( \left\|\int m_{\eps} (\cdot, \x') \diff \rho_0(\x') -1\right\|_{\mathcal{C}\left(\overline{\Omega_0^{\delta}}\right)}\to 0,
\)
which in turn implies
\(
 \left\|\rho_\varepsilon -1\right\|_{\mathcal{C}\left(\overline{\Omega_0^{\delta}}\right)}\to 0
\). 
\end{proof}

\cite{brenier1991polar,CuestaMatran} established the existence and uniqueness (up to additive constants) of a convex function $\phi_\varepsilon$, for every $\varepsilon$, such that its gradient $\nabla \phi_\varepsilon$ pushes forward the measure $\rho_\varepsilon \diff \rho_0$ to $\diff \rho_0$ That is, 
$ \int  f(\nabla \phi_\eps) \, \rho_\eps \diff \rho_0= \int f \,  \diff \rho_0, $ for all $f\in \mathcal{C}(\Omega_0).$
In particular $\phi_\eps$ solves the boundary value problem 
$$ \det(\nabla^2 \phi_\eps)=\frac{\rho_0\, \rho_\eps}{\rho_0(\nabla \phi_\eps) } \quad {\rm in} \ \Omega_0, \quad \nabla \phi_\eps (\Omega_0)= \Omega_0  $$
and its convex conjugate $\phi_\eps^*$ solves 
$$ \det(\nabla^2 \phi_\eps^*)=\frac{\rho_0 } {\rho_0(\nabla \phi_\eps^*)\, \rho_\eps(\nabla \phi_\eps^*)}  \quad {\rm in} \ \Omega_0, \quad \nabla \phi_\eps^* (\Omega_0)= \Omega_0.  $$
Combining Theorem~\ref{Theorem:Caffarelli} with Lemma~\ref{LemmaBoundsUpper}, we can ensure the existence of $\eps_\delta>0$ such that
\begin{equation}\label{Cdeltanotm}
    \|\nabla \phi_\eps\|_{\mathcal{C}^{1, \alpha}(\Omega_0^\delta)}, \ \|\nabla \phi_\eps^*\|_{\mathcal{C}^{1, \alpha}(\Omega_0^\delta)}\leq C(\delta), \quad \text{for all }\eps\leq \eps_\delta.
\end{equation}
Therefore, standard stability theory (see e.g., \cite[Corollary 5.23]{Villani2008}) and the Arzelà–Ascoli theorem yield
\begin{equation}
    \label{stabilty}
    \|\nabla \phi_\eps^*-{\rm Id}\|_{\mathcal{C}^{1}(\Omega_0^\delta)}, \ \|\nabla \phi_\eps-{\rm Id}\|_{\mathcal{C}^{1}(\Omega_0^\delta)} \to 0, \quad \text{as $ \eps\to 0 $.}
\end{equation}
This limit is crucial to show the following result. 

\begin{Lemma}\label{Lemma:developmentDiveps}
    Let Assumption~\ref{Assumptions} holds. There exists $\delta_0$ such that for all  $0<\delta\leq \delta_0$ there exists a function $\omega=\omega_\delta$, such that $\omega(\eps)\to 0$ as $\eps\to 0$, and $C=C(\delta)>0$ such that
    $$ \left\vert  D(\nabla \phi_\eps^*(\x),\nabla g^*(\nabla \phi_\eps^*(\x')))- \|  \x-\x'\|_{\nabla^2 g^*(\x')}^2 \right\vert \leq   C \| \x-\x'\|^{2}( \| \x-\x'\|^{\alpha}+\omega(\eps)),$$
    and 
    $$ \left\vert  \| \nabla \phi_\eps^*(\x) -\nabla \phi_\eps^*(\x')\|_{    \nabla^2 g^*(\nabla \phi_\eps^*(\x'))}^2- \|  \x-\x'\|_{\nabla^2 g^*(\x')}^2 \right\vert \leq   C \| \x-\x'\|^{2}( \| \x-\x'\|^{\alpha}+\omega(\eps)),$$
    for all $\x, \x'\in  \Omega^{\delta}_0=\{ \x\in \Omega_0: \ {\rm dist  }(\x, \partial \Omega_0)\geq \delta \}.$
\end{Lemma}
\begin{proof}
We call $\omega(\eps)= \|\nabla \phi_\eps^*-{\rm Id}\|_{\mathcal{C}^{1}(\Omega_0^\delta)}^{\alpha/2}$, which, by \eqref{stabilty}, tends to $0$ as $\eps\to 0$. 
    Lemma~\ref{LemmaDiver} first and then equation~\eqref{Cdeltanotm} imply that 
\begin{align*}
   & D(\nabla \phi_\eps^*(\x),\nabla g^*(\nabla \phi_\eps^*(\x')))\\
    &= \frac{1}{2}\| \nabla \phi_\eps^*(\x) -\nabla \phi_\eps^*(\x')\|_{    \nabla^2 g^*(\nabla \phi_\eps^*(\x'))}^2 +\mathcal{O}\left( \| \nabla \phi_\eps^*(\x) -\nabla \phi_\eps^*(\x'))\|^{2+\alpha} \right)\\
    &= \frac{1}{2}\| \nabla \phi_\eps^*(\x) -\nabla \phi_\eps^*(\x')\|_{    \nabla^2 g^*(\nabla \phi_\eps^*(\x'))}^2 +\mathcal{O}\left( \| \x-\x'\|^{2+\alpha} \right)
\end{align*}
for all $\x, \x'\in \Omega_0^\delta$. Hence, the proof will be complete after showing that 
\begin{equation}
    \label{theprooffinishes}
     \left \vert \| \nabla \phi_\eps^*(\x) -\nabla \phi_\eps^*(\x')\|_{    \nabla^2 g^*(\nabla \phi_\eps^*(\x'))}^2-\| \x-\x'\|_{    \nabla^2 g^*(\x')}^2\right\vert\leq C \omega(\eps) \| \x-\x'\|^{2} . 
\end{equation}
for all $\x, \x'\in \Omega_0^\delta$ and some $C>0$. 
First we exchange $\nabla^2 g^*(\nabla \phi_\eps^*(\x'))$ by $ \nabla^2 g^*(\x')$. 
Since $\nabla^2 g^*$ is H\"older continuous,  we obtain 
\begin{align*}
   \left \vert \| \nabla \phi_\eps^*(\x) -\nabla \phi_\eps^*(\x')\|_{    \nabla^2 g^*(\nabla \phi_\eps^*(\x'))}^2- \| \nabla \phi_\eps^*(\x) -\nabla \phi_\eps^*(\x')\|_{    \nabla^2 g^*(\x')}^2\right\vert \leq  \omega(\eps) \| \x-\x'\|^{2}. 
\end{align*}
Then we exchange $\nabla \phi_\eps^*$ by the identity 
\begin{align*}
     \| \nabla \phi_\eps^*(\x)& -\nabla \phi_\eps^*(\x')\|_{    \nabla^2 g^*(\x')}^2 \\
     &= \| [\nabla\phi_\eps^*-{\rm Id}](\x) -[\nabla\phi_\eps^*-{\rm Id}](\x')\|_{    \nabla^2 g^*(\x')}^2+ \| \x-\x'\|_{    \nabla^2 g^*(\x')}^2\\
     &\qquad + 2\langle [\nabla\phi_\eps^*-{\rm Id}](\x) -[\nabla\phi_\eps^*-{\rm Id}](\x'),   \x-\x'\rangle_{    \nabla^2 g^*(\x')}\\
     &\leq \| \x-\x'\|_{    \nabla^2 g^*(\x')}^2\left(1+3\omega(\eps)+2\omega(\eps) \right),
\end{align*}
which implies \eqref{theprooffinishes}  and concludes the proof. 
\end{proof}
The following result finds a measure $\mu_\eps$ with the correct marginals and provides the exact description of its density. 

\begin{Lemma}\label{Lemma:Coupling}
 Let Assumption~\ref{Assumptions} hold and   assume further that $\Omega_0$ is convex. Define $$ \mu_\eps:= (\nabla \phi_\eps\times \nabla \phi_\eps)_\#(\xi_{\eps}  \diff (\rho_0 \otimes \rho_0)) $$
 and
 \begin{equation}
    \label{eq:densityOfUeps}
    u_\eps(\x, \x'):= \frac{\diff \mu_\eps}{\diff (\rho_0\otimes  \rho_0)}(\x, \x')=    \frac{\xi_{\eps}(\nabla \phi_\eps^*(\x), \nabla \phi_\eps^*(\x'))  } { \rho_\eps(\nabla \phi_\eps^*(\x))\rho_\eps(\nabla \phi_\eps^*(\x'))}.
\end{equation}
 Then $\mu_\varepsilon$ belongs to $\Pi(\diff \rho_0, \diff \rho_0)$ and, as a consequence, the measure 
$ \pi_\eps^{(1)}=({\rm Id}\times \nabla g )_\# \mu_\eps, $
belongs to $\Pi(\rho_0, \rho_1)$ and it
has density 
$$ \frac{\diff \pi_\eps^{(1)}}{ \diff (\ell_d\otimes \ell_d)}(\x, \y)= u_\eps(\x,\nabla g (\y)) \rho_0(\x) \rho_1(\y).  $$
\end{Lemma}
\begin{proof}
    We can find the shape of its density by computing the following integral for an arbitrary bounded continuous function $f$:
\begin{align*}
    & \int f(\nabla \phi_\eps(\x), \nabla \phi_\eps(\x'))  \xi_{\eps}(\x, \x')  \diff (\rho_0 \otimes \rho_0)(\x, \x')\\
     &= \int \frac{
        f(\x, \x' )  \xi_{\eps}(\nabla \phi_\eps^*(\x), \nabla \phi_\eps^*(\x'))  \rho_0(\nabla \phi_\eps^*(\x)) \rho_0(\nabla \phi_\eps^*(\x') )
     }{
        \big[\det(\nabla^{2} \phi_\eps^*(\x))\det(\nabla^{2} \phi_\eps^*(\x'))\big]^{-1}}  \diff\ell_d(\x) \diff \ell_d(\x')\\
     &=\int \frac{f(\x, \x')  \xi_{\eps}(\nabla \phi_\eps^*(\x), \nabla \phi_\eps^*(\x'))  \rho_0(\nabla \phi_\eps^*(\x)) \rho_0(\nabla \phi_\eps^*(\x') )} {\rho_0(\nabla \phi_\eps^*(\x)) \rho_0(\nabla \phi_\eps^*(\x')) \rho_\eps(\nabla \phi_\eps^*(\x))\rho_\eps(\nabla \phi_\eps^*(\x'))}  \diff\rho_0(\x) \diff \rho_0(\x')\\
     \\
     &=\int \frac{f(\x, \x')  \xi_{\eps}(\nabla \phi_\eps^*(\x), \nabla \phi_\eps^*(\x'))  } { \rho_\eps(\nabla \phi_\eps^*(\x))\rho_\eps(\nabla \phi_\eps^*(\x'))}  \diff\rho_0(\x) \diff \rho_0(\x').
\end{align*}
Since the latter holds for any continuous function $f$,  \eqref{eq:densityOfUeps} holds. 

It is easy to check that by construction 
\begin{align*}
    \int f(\x) u_\eps(\x, \x') \diff\rho_0(\x) \diff \rho_0(\x')&=\int f(\nabla \phi_\eps(\x))  \xi_{\eps}(\x, \x')  \diff (\rho_0 \otimes \rho_0)(\x, \x')\\
    &=\int f(\nabla \phi_\eps(\x))    \diff \rho_\eps(\x)\\
    &=\int f(\x)    \diff \rho_0(\x),
\end{align*}
so that $u_\eps(\x, \x') \diff (\rho_0 \otimes \rho_0)(\x, \x')$ belongs to $\Pi(\rho_0, \rho_0)$. 
\end{proof}
Now we prove some important properties of $u_\eps$.

\begin{Lemma}\label{Lemma:boundonU}
    Let Assumption~\ref{Assumptions} hold and   assume further that $\Omega_0$ is convex. 
Then there exists $\delta_0>0$ such that 
for every $\delta\in (0, \delta_0)$ there exists $\eps_\delta$ such that 
\begin{equation}
    \label{eq:qepsto1}
    \left\| \int_{\Omega^{\delta}_0}  u_\eps(\x, \cdot ) d\rho_0(\x)\right\|_{\mathcal{C}(\Omega^{\delta,c}_0)} \leq \frac{2}{3}
\end{equation}
for all $\eps\leq \eps_\delta$. 
\end{Lemma}
\begin{proof}
Due to \eqref{eq:concentratedSuppoet}, there exists $\delta_0$ such that for all $\delta\leq \delta_0$ there exists  $\eps_\delta$ such that 
$$ \{ \x'\in \Omega_0: \exists \, \x\in \Omega_0^{\frac{\delta}{2}}: \  m_\eps(\x, \x')>0 \}\subset \Omega_{0}^{\frac{\delta}{4}} $$
for all $\eps\leq \eps_\delta$. Equation~\eqref{stabilty} implies that, for a maybe smaller $\eps_\delta$, 
$$ \nabla \phi_\eps^*\left( \Omega_{0}^{\delta} \right)\subset \Omega_{0}^{\frac{\delta}{2}}, \quad \quad \nabla \phi_\eps\left( \Omega_{0}^{\frac{\delta}{4}} \right)\subset \Omega_{0}^{\frac{\delta}{8}}  \quad {\rm and}\quad \nabla \phi_\eps^*\left( \Omega_{0}^{\frac{\delta}{8}} \right)\subset \Omega_{0}^{\frac{\delta}{16}} $$
for all $\eps\leq \eps_\delta$. As a consequence, 
\begin{align*}
    \mathbb{W}_{\eps,\delta}&:=\{ \x'\in \Omega_0: \exists \, \x\in \Omega_0^\delta : \  m_\eps(\nabla \phi_\eps^*(\x), \x')>0 \}\\
 & =   \{ \x'\in \Omega_0: \exists \, \x\in \nabla \phi_\eps^*(\Omega_0^\delta): \  m_\eps(\x, \x')>0 \}\\
 &\subset  \{ \x'\in \Omega_0: \exists \, \x\in \Omega_0^{\frac{\delta}{2}}: \  m_\eps(\x, \x')>0 \}\\
 &\subset \Omega_{0}^{\frac{\delta}{4}}
\end{align*}
and 
\begin{align*}
   \{ \x'\in \Omega_0: \exists \, \x\in \Omega_0^\delta : \  m_\eps(\nabla \phi_\eps^*(\x), \nabla \phi^*(\x'))>0 \}=\nabla \phi (  \mathbb{W}_{\eps,\delta}  ) \subset \nabla \phi ( \Omega_{0}^{\frac{\delta}{4}}  ) \subset \Omega_{0}^{\frac{\delta}{8}}.
\end{align*}
As a consequence, 
$\int_{\Omega^{\delta}_0}  u_\eps(\x, \x' ) d\rho_0(\x)=0$ for every $\x' \in \Omega_{0}^{\frac{\delta}{8},c}$ and $\eps\leq \eps_\delta$. Therefore, we can assume that $\x' \in \Omega_{0}^{\frac{\delta}{8}}$, $\x\in \Omega_0^\delta$ and $\nabla\phi_\eps^*(\x') , \nabla\phi_\eps^*(\x) \in \Omega_{0}^{\frac{\delta}{16}}$. Due to Lemma~\ref{LemmaBoundsUpper},~\ref{LemmaBoundsUpperPoint4},
\begin{align}\label{eq:boundUepsLemmaproof}
    u_\eps(\x, \x')= \frac{\xi_{\eps}(\nabla \phi_\eps^*(\x), \nabla \phi_\eps^*(\x'))  } { 1+o(1) }= \frac{m_{\eps}(\nabla \phi_\eps^*(\x), \nabla \phi_\eps^*(\x'))  } { 1+o(1) },
\end{align}
where the error term is uniform for $\x, \x'\in \Omega_{0}^{\frac{\delta}{8}}$. On the other hand,
\begin{equation}
    \label{eq:boundModulusCepsphi}
\vert C_\eps(\nabla \phi_\varepsilon^*(\x))-C_\eps(\x)\vert  \leq \eps^{\frac{2}{d+2}}  \|\nabla \phi_\eps^*-{\rm Id}\|_{\mathcal{C}^{1}(\Omega_0^\delta)}^{\alpha} =o\left( \eps^{\frac{2}{d+2}} \right), 
\end{equation}
where the last term is a consequence of \eqref{stabilty}.  We observe that for $\x, \x'\in  \Omega_{0}^{\frac{\delta}{8}}$ such that the inequality 
$
m_\eps(\nabla \phi_\eps^*(\x), \nabla \phi_\eps^*(\x'))>0
$ holds,
equation \eqref{eq:boundCeps} implies
$\|\nabla \phi_\eps^*(\x)-\nabla \phi_\eps^*(\x')\|^2 \leq C\, \eps^{\frac{2}{d+2}}, $
and, by \eqref{stabilty}, also 
\begin{equation}\label{boundsoporteLemmaU}
    \|\x-\x'\|^2 \leq C\, \eps^{\frac{2}{d+2}}.
\end{equation}
 Hence, \eqref{boundHolderg} yields 
$$ \|\x-\x'\|^2_{\nabla^2 g^*(\x)}= \|\x-\x'\|^2_{\nabla^2 g^*(\x')}  + o\left( \eps^{\frac{2}{d+2}} \right) , $$
where the error term is uniform for $\x, \x'\in  \Omega_{0}^{\frac{\delta}{8}}$ such that  $ m_\eps(\nabla \phi_\eps^*(\x), \nabla \phi_\eps^*(\x'))>0$. Hence, the previous display, \eqref{boundsoporteLemmaU} and \eqref{eq:boundModulusCepsphi} imply 
$$ m_\eps(\nabla \phi_\eps^*(\x), \nabla \phi_\eps^*(\x')) = \frac{1}{\eps} \left( C_\eps(\x')-\frac{1}{2}\|\x-\x'\|^2_{\nabla^2 g^*(\x')}+ o\left( \eps^{\frac{2}{d+2}} \right)\right)_+ $$
for all $\x, \x'\in  \Omega_{0}^{\frac{\delta}{8}}$ such that  $ m_\eps(\nabla \phi_\eps^*(\x), \nabla \phi_\eps^*(\x'))>0$.
This,  combined with \eqref{eq:boundUepsLemmaproof}, yields
\begin{align*}
     &\int_{\Omega_0^{\delta}}  u_\eps(\x, \x') \diff \rho_0(\x)\\ &= \frac{ \int_{\Omega_0^{\delta}}  m_\eps(\nabla \phi_\eps^*(\x), \nabla \phi_\eps^*(\x')) \diff \rho_0(\x)  } { 1+o(1) }\\
     &\leq \frac{1}{\eps} \frac{  \int_{\Omega_0^{\delta}}  \left( C_\eps(\x')-\frac{1}{2}\|\x-\x'\|^2_{\nabla^2 g^*(\x')}+ o\left( \eps^{\frac{2}{d+2}} \right)\right)_+ \diff \rho_0(\x)  } { 1+o(1) }\\
     \\
     &\leq \frac{1}{\eps} \frac{  (\rho_0(\x')+o(1))\int_{\Omega_0^{\delta}}  \left( C_\eps(\x')-\frac{1}{2}\|\x-\x'\|^2_{\nabla^2 g^*(\x')}+ o\left( \eps^{\frac{2}{d+2}} \right)\right)_+ \diff \ell_d(\x)  } { 1+o(1) }.
\end{align*}
{Changing} variables $\u= T(\x)= [\nabla^2 g^*(\x')]^{\frac{1}{2}} (\x-\x')$ to get (up to uniform $o(1)$ terms) 
$$ \int_{\Omega_0^{\delta}}   u_\eps(\x, \x') \diff \rho_0(\x) \leq  \frac{1}{\varepsilon}\cdot\frac{\rho_0(\x')}{\det(\nabla^2 g^*(\x'))^{1/2}}\int_{T(\Omega_0^{\delta})}  \left( C_\eps(\x')-\frac{\|\u\|^2}{2}\right)_+ \diff \ell_d({\u}).  $$
Without losing generality we can assume that $\x'={\bf 0}$. Then ${\bf 0}\notin T(\Omega_0^{\delta})$. Since $ T(\Omega_0^{\delta})$ is convex, there exist a vector $\v$ and a  hyperplane $H=\{ \z: \langle \z, \v\rangle=0 \} $ such that $T(\Omega_0^{\delta})\subset  H_+:=\{ \z: \langle \z, \v\rangle\leq 0 \} $. As a consequence, up to $o(1)$ additive terms,
\begin{align*}
    \int_{\Omega_0^{\delta}}   u_\eps(\x, \x') \diff \rho_0(\x) &\leq \frac{1}{\varepsilon}\cdot\frac{\rho_0(\x')}{\det(\nabla^2 g^*(\x'))^{1/2}}\int_{H_+}  \left( C_\eps(\x')-\frac{\|\u\|^2}{2}\right)_+ \diff \ell_d({\u})\\
    &= \frac{1}{\varepsilon}\cdot\frac{\rho_0(\x')}{2\,\det(\nabla^2 g^*(\x'))^{1/2}}\int   \left( C_\eps(\x')-\frac{\|\u\|^2}{2}\right)_+ \diff \ell_d({\u})\\
   \text{(by Proposition~\ref{Lemma:elementary}) } &=\frac{1}{\varepsilon}\cdot\frac{\rho_0(\x')}{2\,\det(\nabla^2 g^*(\x'))^{1/2}}|C_\eps(\x')|^{\frac{d+2}{2}} C_d \\
   &=\frac{1}{\varepsilon}\cdot\frac{\rho_0(\x')}{2\,\det(\nabla^2 g^*(\x'))^{1/2}}\left( \frac{\eps^{\frac{2}{d+2}}}{C_d^{\frac{2}{d+2}} \big(\rho_0(\x')\rho_1[\nabla g^*(\x')]\big)^{\frac{1}{d+2}}}\right)^{\frac{d+2}{2}} C_d\\
   &=\frac{\rho_0(\x')}{2\,\det(\nabla^2 g^*(\x'))^{1/2}} \frac{1}{ \big(\rho_0(\x')\rho_1[\nabla g^*(\x')]\big)^{\frac{1}{2}}}.
\end{align*}
Since, $\det(\nabla^2 g^*)=\frac{\rho_0}{\rho_1(\nabla g^*) }$, we get 
$\int_{\Omega_0^{\delta}}   u_\eps(\x, \x') \diff \rho_0(\x)\leq \frac{1}{2}+o(1), $
and the claim follows. 

\end{proof}

\subsection{Valid coupling for the frame $\Omega_0^{\delta,c}$ of a convex $\Omega_0$}\label{sec:fixing_the_coupling}
For  $\delta>0$, we recall the definition \eqref{def:omega_delta} and consider the truncated function
\[
u_\eps \chi_{(\Omega^{\delta}_0\times \Omega_0) \cup ( \Omega_0\times \Omega^{\delta}_0)},
\]
in an effort to avoid the problematic regions of $\Omega_0\times \Omega_0$ where the candidate for coupling is not regular enough. Notice how we allow one of the variables, $\x$ for example, to be close to the boundary as long as the other one, $\x'$, is not. In principle, this may seem inadvisable; however, there is no issue since the support of the coupling will concentrate around the set $\x=\x'$. Refer to the proof of Lemma~\ref{Lemma:boundonU}  above or that of Lemma~\ref{Lemma:limsup} below for further details. Therefore,  the only truly problematic set is $(\Omega^{\delta, c}_0\times\Omega^{\delta, c}_0)$.

Thus, in seeking a true coupling in$(\Omega_0\times\Omega_0)$ we search for a density $v_\eps$ such  that
$$ 1=\int  v_\eps(\x, \x') \diff \rho_0(\x)= \int  v_\eps(\x, \x') \diff \rho_0(\x'), \quad \text{for all } \x, \x'\in \Omega_0,$$
and we want it of the form $$v_\eps= u_\eps \chi_{(\Omega^{\delta}_0\times \Omega_0) \cup ( \Omega_0\times \Omega^{\delta}_0)}+  h_{\varepsilon, \delta}\chi_{(\Omega^{\delta,c}_0\times \Omega^{\delta,c}_0 )}, $$ 
for some positive function $h_{\eps, \delta}$. Irrespective of $ h_{\varepsilon, \delta}$, for $\x'\in  \Omega^{\delta}_0$ the ``coupling'' condition 
$$
1= \int v_\eps(\x,\x') d\rho_0(\x)
$$
must hold. Hence, we need only to fit $h_{\varepsilon, \delta}$ such that 
that 
$$ 1= \int_{\Omega^{\delta,c}_0}   h_{\varepsilon, \delta}(\x, \x') d\rho_0(\x) + \int_{\Omega^{\delta}_0}  u_\eps(\x, \x') d\rho_0(\x), \quad \text{for all }\x'\in  \Omega^{\delta,c}_0 $$
and 
\begin{equation}
    \label{deltaCondition}
    1= \int_{\Omega^{\delta,c}_0}   h_{\varepsilon, \delta}(\x, \x') d\rho_0(\x') + \int_{\Omega^{\delta}_0}  u_\eps(\x, \x') d\rho_0(\x'), \quad \text{for all }\x\in  \Omega^{\delta,c}_0.
\end{equation}
We find first $h_{\eps, \delta}$ such that condition \eqref{deltaCondition} holds and then we construct a coupling based on this. In order to simplify the notation we define
$$ 
    q_\delta(\x):=1-\left(\int_{\Omega^{\delta}_0}  u_\eps(\x, \x') d\rho_0(\x')\right)\chi_{\Omega^{\delta,c}_0}(\x) \in (0,1]. 
$$
Note that due to Lemma~\ref{Lemma:boundonU}, we can assume that $\varepsilon$ is sufficiently small so that
\begin{equation}
    \label{qepsBound}
    \frac13 \leq q_\delta(\mathbf{x}) \leq 1, \quad \text{for all }\x\in \Omega^{\delta, c}.
\end{equation}
Next, we show that \eqref{ConeCond} holds for $\Omega^{\delta, c}$.
\begin{Lemma}\label{Lemma:CondConeCehck}
    Let Assumption~\ref{Assumptions} hold and $\Omega_0$ be convex.  there exist $ \delta_0>0$ and $\theta\in (0,\pi)$ such that for all $\delta\leq \delta_0$ and $\x\in \Omega^{\delta,c}$ there exists a cone
    $\mathcal{C}_{\x, {\bf v},r, \theta} $
    with vertex $\x$,  height $r=r_{\delta}>0$ and angle $\theta$ such that $\mathcal{C}_{\x, {\bf v},r, \sigma}\subset \Omega^{\delta,c}$.  
\end{Lemma}
\begin{proof}
    Let $\delta_0$ be such that $\Omega_0^{\delta}$ exists and is convex and nonempty,  $\delta\leq \delta_0$ and  $r'_\delta= \delta/4$. Then, for $\x \in \Omega^{\delta,c}_0$ and \( r' \leq r_\delta' \), the ball \( \mathbb{B}\left(\mathbf{x}, r'\right) \) intersects at most one of \( \partial \Omega_0^{\delta} \) or \( \partial \Omega_0 \), but not both simultaneously. In the last case, the claim holds by Assumption~\ref{Assumptions}. Therefore, we can assume that \( \mathbb{B}\left(\mathbf{x}, r'\right) \cap  \partial \Omega_0=\emptyset \). 
 Since $\Omega_0^{\delta}$ is convex and $ \x\notin \Omega_0^{\delta} $, there exists a separating half-space $H_+=\{\z: \langle \z, {\bf v}\rangle\geq a\} $ such that $\x\in H_+ $ and $\Omega_{0}^{\delta}\cap  H_+=\emptyset $. Therefore, 
 $$ \mathbb{B}\left(\mathbf{x}, r'  \right) \cap H_+\subset \mathbb{B}\left(\mathbf{x}, r  \right) \cap \Omega^{\delta,c}_0,$$
for every $r'\leq r_\delta$. Setting $\theta=\pi/4$ and 
$$ r=\max\left\{ x_1: \ \|(x_1, \dots, x_d)\|\leq r_\delta' \quad  {\rm and}\quad  \|(x_1, \dots, x_d)\|  \cos\left(\frac{1}{2}\theta\right) \leq 
x_1 \right\}>0$$
the result follows. 
\end{proof}
Due to Lemma~\ref{Lemma:CondConeCehck} and \eqref{qepsBound}, the conditions of Lemma~\ref{Lemma:firstEstimates} hold for $s=q_\delta$ and $\Omega=\Omega^{\delta,c}$. As a consequence, for $\eps$ small enough there exists $c_{\eps,\delta}:\Omega^{\delta,c} \to \R $ such that
$$q_\delta(\x)=\frac{1}{\eps}\int_{\Omega^{\delta,c}_0} (c_{\eps,\delta} (\x)-D(\x, \nabla g^*(\x') )_+\diff \rho_0(\x') , \quad \text{for all } \x\in \Omega^{\delta,c}_0.    $$
Moreover,
there exists $C=C(\delta_0)$ such that for every  $\delta\leq \delta_0$ there exists $\eps_\delta>0$  such that for   $\eps\leq \eps_\delta$, the function $c_{\eps, \delta}$ exists,  

     \begin{align}
         \label{eq:maxcepsCeps} &\varepsilon^{\frac{2}{d+2}} C^{-1} \leq  c_{\eps, \delta}(\x) \leq \varepsilon^{\frac{2}{d+2}} C,  \quad \varepsilon^{\frac{d}{d+2}} C^{-1}\leq \ell_d(\mathcal{S}_{\varepsilon,\delta, \x})\leq C\varepsilon^{\frac{d}{d+2}}, 
 \quad \text{and}\\
    &  \frac{1}{\eps}\int_{\Omega^{\delta,c}_0} (c_{\eps, \delta} (\x)-D(\x, \nabla g^*(\x') )_+\diff \rho_0(\x) \geq C^{-1}, \label{eq:maxcepsqdelta}
     \end{align}
     for every  $\x'\in \Omega_0^{\delta,c}$ 
where $\mathcal{S}_{\varepsilon,\delta, \x}= \{ \x'\in \Omega_0:  c_{\eps, \delta} (\x)-D(\x, \nabla g^*(\x')  \geq 0\}.$

To simplify the exposition, we first introduce a few recurrent quantities. Define
\[ 
M_{\varepsilon, \delta} (\mathbf{a}, \mathbf{b}) := \frac{(c_{\varepsilon, \delta}(\mathbf{a}) - D(\mathbf{a}, \nabla g^*(\mathbf{b}))_+}{\varepsilon},
\]
and the density 
\[
h_{\varepsilon, \delta}(\mathbf{x}, \mathbf{x}') := \int_{\Omega_0^{\delta,c}} \frac{M_{\varepsilon, \delta}(\mathbf{x}', \mathbf{z}) M_{\varepsilon, \delta}(\mathbf{x}, \mathbf{z})}{\int_{\Omega_0^{\delta,c}} M_{\varepsilon, \delta}(\mathbf{v}', \mathbf{z}) \, d\rho_0(\mathbf{v}')} \, d\rho_0(\mathbf{z}),
\]
with the associated measure given by
\[
d\pi_\varepsilon^{(2)}(\mathbf{x}, \mathbf{x}') = h_{\varepsilon, \delta}(\mathbf{x}, \mathbf{x}') \, d\rho_0(\mathbf{x}) \, d\rho_0(\mathbf{x}').
\]
\begin{Lemma}\label{Lemma:integraldaq}
 Let Assumption~\ref{Assumptions} hold. Then  
    $$ q_\delta(\x)= \int_{\Omega^{\delta,c}_0}  h_{\varepsilon, \delta}(\x, \x')\diff\rho_0(\x') \quad {\rm and } \quad q_\delta(\x')= \int_{\Omega^{\delta,c}_0}   h_{\varepsilon, \delta}(\x, \x')\diff\rho_0(\x). $$
    As a consequence, 
    $v_\eps \diff (\rho_0 \otimes \rho_0)\in \Pi(\rho_0, \rho_0)  $,
    where
    $$
    v_\eps= [u_\eps \chi_{(\Omega^{\delta}_0\times \Omega_0) \cup ( \Omega_0\times \Omega^{\delta}_0)}+ h_{\eps, \delta} \chi_{(\Omega^{\delta,c}_0\times \Omega^{\delta,c}_0 )} ]
    $$
\end{Lemma}
\begin{proof}
    Set $\x$
    and compute 
    \begin{align*}
\int_{\Omega^{\delta,c}_0}  h_{\varepsilon, \delta}(\x, \x')\diff\rho_0(\x')&= \int_{\Omega^{\delta,c}_0}  \int_{\Omega^{\delta,c}_0}{\frac{ M_{\eps, \delta} (\x', \z)  M_{\eps, \delta} (\x, \z)  }{\int_{\Omega^{\delta,c}_0} M_{\eps, \delta}(\v', \z) \diff \rho_0(\v') }} \diff \rho_0(\z)  \diff\rho_0(\x')\\
&=   \int_{\Omega^{\delta,c}_0} {\frac{ \int_{\Omega^{\delta,c}_0}M_{\eps, \delta} (\x', \z)  \diff\rho_0(\x') M_{\eps, \delta} (\x, \z)  }{\int_{\Omega^{\delta,c}_0} M_{\eps, \delta}(\v', \z) \diff \rho_0(\v') }} \diff \rho_0(\z)  \\
&= \int_{\Omega^{\delta,c}_0} { M_{\eps, \delta} (\x, \z)  } \diff \rho_0(\z) \\
&=q_{\delta}(\x).
    \end{align*}
The same holds for the other variable by symmetry. 
\end{proof}
We give some estimates on the function $ h_{\eps, \delta}(\x, \x')$, which will be useful for the sequel.  
\begin{Lemma}\label{Lemma:BoundH}
    Let assumption~\ref{Assumptions} hold. Then there there exists  $\delta_0>0$ and $C>0$ such that for every $\delta\in (0, \delta_0)$  there exist $ \eps_\delta$
such that  
\begin{equation}
    \label{BoundforDivLemma48}
    \int_{\Omega^{\delta,c}_0\times \Omega^{\delta,c}_0} D(\x,\nabla g^*(\x')) h_{\eps, \delta}(\x, \x')  \diff (\rho_0\otimes \rho_0)(\x, \x') \leq C  \ell_d(\Omega_0^{\delta,c} )\,\eps^{\frac{2}{d+2}} , 
\end{equation}
and 
\begin{equation}\label{BoundOnheps}
\frac{\eps}{2}\int_{\Omega^{\delta,c}_0\times \Omega^{\delta,c}_0}   (h_\eps(\x, \x'))^2\diff (\rho_0\otimes \rho_0)(\x, \x')  \leq C\, \ell_d(\Omega_0^{\delta,c} )\, \eps^{\frac{2}{d+2}} , 
\end{equation}
for all $\eps
\leq \eps_\delta$.
 
\end{Lemma}
\begin{proof}
By \eqref{eq:maxcepsqdelta},  it is sufficient to verify \eqref{BoundforDivLemma48} and \eqref{BoundOnheps} by substituting  $h_{\eps, \delta}$ by  $M_{\eps, \delta}$. 
First we give some estimates on the support of $M_{\eps,\delta}$. The set of $(\x, \x')$ such that  $\int_{\Omega^{\delta,c}_0} M_{\eps, \delta} (\x', \z)  M_{\eps, \delta} (\x, \z)  \diff \rho_0(\z)>0$ is called 
$A_{\delta, \eps}$ . It can be also described as the set of $(\x, \x')$ such that  there exists $\z\in \Omega_{0}^{\delta,c}$ belonging to ${\rm supp}(M_{\eps, \delta} (\x, \cdot ))\cap {\rm supp}(M_{\eps, \delta} (\x', \cdot )) $, i.e., $\z$ satisfying
$ M_{\eps, \delta} (\x', \z)  M_{\eps, \delta} (\x, \z)>0.  $
Recall that 
$\int_{\Omega^{\delta,c}_0} M_{\eps, \delta} (\x', \z)  M_{\eps, \delta} (\x, \z)  \diff \rho_0(\z)>0$. 
\eqref{eq:maxcepsCeps} implies that there exists $C>0$ such that 
$ {\rm supp}(M_{\eps, \delta} (\x, \cdot ))\subset  \mathbb{B}\left(\x, C\eps^{\frac{2}{d+2}}\right),$
so that 
$$A_{\delta,\eps}\subset \left\{ (\x, \x')\in (\Omega^{\delta,c}_0)^2: \  \mathbb{B}\left(\x, C\eps^{\frac{2}{d+2}}\right)\cap \mathbb{B}\left(\x', C\eps^{\frac{2}{d+2}}\right)\neq \emptyset\right\}.$$
Therefore, by increasing \(C\), we have 
\begin{equation}
    \label{BallContainedheps}
    A_{\delta,\eps}\subset \left\{ (\x, \x')\in (\Omega^{\delta,c}_0)^2: \  \x'\in \mathbb{B}\left(\x, C\eps^{\frac{2}{d+2}}\right)\right\}.
\end{equation}
Now we give some estimates on the maximum value that $h_{\eps, \delta}$ can take. 
According to  \eqref{eq:maxcepsCeps},  there exists $\delta_0>0$ and $C(\delta_0)>0$ such that for every $\delta\in (0, \delta_0)$  there exist $ \eps_\delta$
such that for every $\eps\leq \eps_\delta$, it holds that 
$ M_{\eps, \delta} (\x', \z)  \leq   C(\delta_0)\eps^{\frac{-d}{d+2}}. $ 
Hence,  
\begin{align*}
     \int_{\Omega^{\delta,c}_0} M_{\eps, \delta} (\x', \z) M_{\eps, \delta} (\x, \z) d\rho_0(\z)& \leq C(\delta_0)\eps^{\frac{-d}{d+2}} \int_{\Omega^{\delta,c}_0} M_{\eps, \delta} (\x, \z) \diff\rho_0(\z)\\
     & = C(\delta_0)\eps^{\frac{-d}{d+2}} q_\delta(\x)\\
     & \leq  C(\delta_0)\eps^{\frac{-d}{d+2}}, 
\end{align*}
which implies 
\begin{equation}
    \label{boundheps}
    \max_{\z}{h_{\eps, \delta}(\x', \z)}\leq C(\delta_0) \eps^{-\frac{d}{d+2}}.
\end{equation}
At this point, we have all the ingredients to prove \eqref{BoundforDivLemma48} and \eqref{BoundOnheps}. Firstly,  
we find an upper bound for \eqref{BoundOnheps} as follows: 
\begin{align*}
    \frac{\eps}{2}\int_{\Omega^{\delta,c}_0\times \Omega^{\delta,c}_0} &  (h_\eps(\x, \x'))^2\diff (\rho_0\otimes \rho_0)(\x, \x') \\
    &\leq   \frac{\eps}{2} \int_{\Omega^{\delta,c}_0} \max_{\z}{h_{\eps, \delta}(\x', \z)} \int_{\Omega^{\delta,c}_0}  h_{\eps, \delta}(\x, \x')\diff \rho_0(\x) \diff\rho_0(\x')  \\
    &= \frac{\eps}{2} \int_{\Omega^{\delta,c}_0} \max_{\z}{h_{\eps, \delta}(\x, \z)} q_\delta(\x)\diff  \rho_0(\x') \\
    &\leq \frac{\eps}{2} \int_{\Omega^{\delta,c}_0} \max_{\z}{h_{\eps, \delta}(\x', \z)}\diff \rho_0(\x') \\
\text{(by \eqref{boundheps}) }    &\leq  \frac{C(\delta_0) \eps^{\frac{2}{d+2}}}{2}  \Lambda \ell_d(\Omega^{\delta,c}_0).
\end{align*}
Hence, \eqref{BoundOnheps} holds. To prove \eqref{BoundforDivLemma48}, we use first Lemma~\ref{LemmaDiver} and then \eqref{BallContainedheps}  to obtain 
\begin{align*}
    \int_{\Omega^{\delta,c}_0\times \Omega^{\delta,c}_0}& D(\x,\nabla g^*(\x')) h_{\eps, \delta}(\x, \x') \diff (\rho_0\otimes \rho_0)(\x, \x') \\&\leq \frac{\sigma_M(g)}{2} \int_{\Omega^{\delta,c}_0\times \Omega^{\delta,c}_0} \|\x-\x'\|^2 h_{\eps, \delta}(\x, \x')  \diff (\rho_0\otimes \rho_0)(\x, \x')\\
    &\leq C(\delta_0, g) \eps^{\frac{2}{d+2}} \int_{\Omega^{\delta,c}_0\times \Omega^{\delta,c}_0}  h_{\eps, \delta}(\x, \x')  \diff (\rho_0\otimes \rho_0)(\x, \x').
\end{align*}
Since, via Lemma~\ref{Lemma:integraldaq}, we have
$ \int_{\Omega^{\delta,c}_0} h_{\eps, \delta}(\x, \x')  \diff \rho_0(\x) =q_\delta(\x') \leq 1,  $
wich yields
\begin{align*}
    \int_{\Omega^{\delta,c}_0\times \Omega^{\delta,c}_0}& D(\x,\nabla g^*(\x')) h_{\eps, \delta}(\x, \x') \diff (\rho_0\otimes \rho_0)(\x, \x') \leq C(\delta_0, g) \eps^{\frac{2}{d+2}} \Lambda \ell_d\left(\Omega^{\delta,c}_0 \right).
\end{align*}
The result follows. 
\end{proof}

\subsection{Limit for convex supports}

Let us define the measure \(\nu_\varepsilon(\x,\x') = v_\eps(\x, \x')  \diff (\rho_0\otimes \rho_0)(\x, \x'),\)
meaning that if we define $\pi_\varepsilon:=(Id \times \nabla g^*)_\# \nu_\eps$ then we readily obtain
\[
\pi_\eps(x,y) =(Id \times \nabla g^*)_\# \nu_\eps=v_\eps(\x, \nabla g(\y))  \diff (\rho_0\otimes \rho_1)(\x, \y).
\]

In this section, we find the limit of $ \eps^{-\frac{2}{d+2}} (\mathbb{H}_\eps ( \pi_{\eps})-\mathcal{W}_2^2(\rho_0, \rho_1)), $
where 
$$ \mathbb{H}_\eps ( \pi)= \int \|\x-\y\|^2 \pi(\x,\y) \diff  (\rho_0\otimes \rho_1)(\x,\y) + \eps\, \left\|\frac{\diff \pi}{\diff (\rho_0\otimes\rho_1)}\right\|_{L^2(\rho_0\otimes\rho_1)}^2, $$
  in the case where $\Omega_0$ is a convex domain, which will be provided by the following asymptotic bound.  
\begin{Lemma}\label{Lemma:limsup}
    Let Assumption \ref{Assumptions} hold and $\Omega_0$ be convex. Then 
    $$ \limsup_{\eps\to 0^+} \frac{\mathcal{T}_{2, \varepsilon, (\cdot)^2}(\rho_0, \rho_1) - \mathcal{W}_2^2(\rho_0, \rho_1)}{\eps^{\frac{2}{d+2}}}\leq \frac{d^{\frac{d+4}{d+2}}(d+2)^{\frac{2}{d+2}}}{\left(\mathcal{H}^{d-1}(\mathcal{S}^{d-1}) \right)^{\frac{2}{d+2}}}   \int \big(\rho_0(\x)\rho_1[\nabla g^*(\x)]\big)^{-\frac{1}{(d+2)}}\diff \rho_0(\x). $$   
\end{Lemma}
\begin{proof}
First, as in the proof of Lemma~\ref{Lemma:boundonU}, there exists $\delta_0$ such that for all $\delta\leq \delta_0$ there exists  $\eps_\delta$ such that  
\begin{align}
  & \{ \x'\in \Omega_0: \exists \, \x\in \Omega_0^\delta : \  m_\eps(\nabla \phi_\eps^*(\x), \nabla \phi^*(\x'))>0 \} \subset \Omega_{0}^{\frac{\delta}{8}}, \label{contentionSuppo}
       \\      \label{contentionnablaphi}
&\nabla \phi_\eps^*\left( \Omega_{0}^{\delta} \right)\subset \Omega_{0}^{\frac{\delta}{2}}, \quad \quad \nabla \phi_\eps^*\left( \Omega_{0}^{\frac{\delta}{4}} \right)\subset \Omega_{0}^{\frac{\delta}{8}}  \quad {\rm and}\quad \nabla \phi_\eps^*\left( \Omega_{0}^{\frac{\delta}{8}} \right)\subset \Omega_{0}^{\frac{\delta}{16}} 
\end{align}
for all $\eps\leq \eps_\delta$. Therefore, from now on we assume that $\eps\leq \eps_\delta$ for $0<\delta\leq \delta_0$. 
We analyze first the linear term, where $\pi_0$ is the minimizer of \eqref{eq: OT}. By duality \eqref{eq:OTdual}, it holds that 
\begin{multline*}
    \frac12 \int \|\x-\y\|^2 \diff (\pi_\eps-\pi_0)(\x,\y)\\=  \int \frac12\|\x-\y\|^2 \diff\pi_\eps(\x,\y) - \int \left(\frac12\|\x\|^2-g^*(\x)\right) \diff \rho_0(\x)-\int \left(\frac12\|\y\|^2-g(\y)\right)  \diff \rho_1(\y).
\end{multline*}
Since $\pi_\eps\in \Pi(\rho_0, \rho_1)$, we get 
\begin{align*}
    \frac12 \int \|\x-\y\|^2 \diff (\pi_\eps-\pi_0)(\x,\y)&=\int D(\x, \y) \diff \pi_\eps(\x,\y)
    \\
    &=\int D(\x, \nabla g^*(\x')) \diff \nu_\eps (\x,\x').
\end{align*}
We observe that 
\begin{align*}
    &\int D(\x,\nabla g^*(\x')) \diff \nu_\eps(\x,\x') \\
    &= \underbrace{\int_{(\Omega^{\delta}_0\times \Omega_0) \cup ( \Omega_0\times \Omega^{\delta}_0)} D(\x,\nabla g^*(\x')) u_\eps(\x, \x')  \diff (\rho_0\otimes \rho_0)(\x, \x')}_{\mathbb{L}_1}\\
    &\qquad + \underbrace{\int_{(\Omega^{\delta,c}_0\times \Omega^{\delta,c}_0)} D(\x,\nabla g^*(\x')) h_{\eps, \delta}(\x, \x')  \diff (\rho_0\otimes \rho_0)(\x, \x')}_{\mathbb{L}_2} .
\end{align*}
We denote the first term by \(\mathbb{L}_1\) and the second by \(\mathbb{L}_2\). To bound \(\mathbb{L}_2\), Lemma~\ref{Lemma:BoundH} ensures that there exist \(\delta_0 > 0\) and \(C > 0\) such that for every \(\delta \in (0, \delta_0)\), there is an \(\varepsilon_\delta\) satisfying   
\begin{equation}
    \label{eq:boundLinearBAd}
\mathbb{L}_2=\int_{(\Omega^{\delta,c}_0\times \Omega^{\delta,c}_0)} D(\x,\nabla g^*(\x')) h_{\eps, \delta}(\x, \x')  \diff (\rho_0\otimes \rho_0)(\x, \x')\leq C\rho_0(\Omega_0^{\delta}) \eps^{\frac{2}{d+2}}
\end{equation}
for every $\eps\leq\eps_\delta$.
For $\mathbb{L}_1$ we observe that Lemma~\ref{LemmaBoundsUpper} (point  \ref{LemmaBoundsUpperPoint4}), \eqref{contentionSuppo} and \eqref{contentionnablaphi}, imply  that
$$ \mathbb{L}_1\leq  \frac{1}{1+o(1)}\int_{(\Omega^{\delta'}_0\times \Omega^{\delta'}_0)} D(\nabla \phi_\eps^*(\x),\nabla g^*(\nabla \phi_\eps^*(\x'))) m_\eps(\x, \x')  \diff (\rho_0\otimes \rho_0)(\x, \x'),  $$
for some $0<\delta'< \delta$. 
By 
Lemma~\ref{Lemma:developmentDiveps}, we have 
$$ \left\vert  D(\nabla \phi_\eps^*(\x),\nabla g^*(\nabla \phi_\eps^*(\x')))- \frac{1}{2}\|  \x-\x'\|_{\nabla^2 g^*(\x')}^2 \right\vert \leq   C \| \x-\x'\|^{2}( \| \x-\x'\|^{\alpha}+\omega(\eps)),  $$
for all $\x, \x'\in \Omega_0^{\delta'}$,
where $\omega(\eps)\to 0$ as $\eps\to 0$.
As a consequence, up to $o(1)$ terms, we have the bound
\begin{align*}
      \mathbb{L}_1&\leq \int_{(\Omega^{\delta'}_0)^2} \frac{1}{2}\|  \x-\x'\|_{\nabla^2 g^*(\x')}^2 m_\eps(\x, \x')  \diff (\rho_0\otimes \rho_0)(\x, \x')\\
        &\qquad + C\,\int \frac{1}{2}\|  \x-\x'\|^2 ( \| \x-\x'\|^{\alpha}+\omega(\eps))  m_\eps(\x, \x')  \diff (\rho_0\otimes \rho_0)(\x, \x'),
\end{align*}
where, due to \eqref{eq:concentratedSuppoet}, the last term is of order  $o(\varepsilon^{\frac{2}{d+2}})$. Therefore, it holds that
\begin{align*}
     \mathbb{L}_1  \le   \int_{(\Omega^{\delta'}_0)^2} \frac{1}{2}\|  \x-\x'\|_{\nabla^2 g^*(\x')}^2 m_\eps(\x, \x')  \diff (\rho_0\otimes \rho_0)(\x, \x') + o(\varepsilon^{\frac{2}{d+2}}). 
\end{align*}

 Lemma~\ref{lemma:symetric} yields the estimate
 \begin{equation}
     \label{lemma:symetricForProof} 
     m_\eps(\x, \x')=\frac{1}{\eps}\left( C_\eps(\x') -\frac{1}{2}\|  \x-\x'\|_{\nabla^2 g^*(\x')}^2\right)_+ + \omega(\eps) \eps^{\frac{-d}{d+2}}, 
 \end{equation}
which implies
\begin{multline*}
    \mathbb{L}_1 
    \leq  \eps^{-1}\int_{(\Omega^{\delta'}_0)^2} \frac{1}{2}\|  \x-\x'\|_{\nabla^2 g^*(\x')}^2 \left( C_\eps(\x') -\frac{1}{2}\|  \x-\x'\|_{\nabla^2 g^*(\x')}^2\right)_+  \diff (\rho_0\otimes \rho_0)(\x, \x')\\+ \omega(\eps) \eps^{\frac{-d}{d+2}} \int_{ {\rm supp}(m_\eps)} \frac{1}{2}\|  \x-\x'\|_{\nabla^2 g^*(\x')}^2  \diff (\rho_0\otimes \rho_0)(\x, \x')
    +o\left(\eps^{\frac{2}{d+2}}\right).
\end{multline*}
The bound \eqref{eq:concentratedSuppoet} implies 
\begin{multline*}
    \mathbb{L}_1 
    \leq  \eps^{-1}\int_{(\Omega^{\delta'}_0)^2} \frac{1}{2}\|  \x-\x'\|_{\nabla^2 g^*(\x')}^2 \left( C_\eps(\x') -\frac{1}{2}\|  \x-\x'\|_{\nabla^2 g^*(\x')}^2\right)_+  \diff (\rho_0\otimes \rho_0)(\x, \x')\\
    +o\left(\eps^{\frac{2}{d+2}}\right).
\end{multline*}
By adding and subtracting $ C_\eps(\x')$ inside the integral, the bound
\begin{align*}
   \mathbb{L}_1 &\leq    -\eps^{-1}\int_{(\Omega^{\delta'}_0)^2}  \left( C_\eps(\x') -\frac{1}{2}\|  \x-\x'\|_{\nabla^2 g^*(\x')}^2\right)_+^2 \diff (\rho_0\otimes \rho_0)(\x, \x')\\
     &\qquad
     + \underbrace{\eps^{-1}\int_{(\Omega^{\delta'}_0)^2} C_\eps(\x') \left( C_\eps(\x') -\frac{1}{2}\|  \x-\x'\|_{\nabla^2 g^*(\x')}^2\right)_+ \diff (\rho_0\otimes \rho_0)(\x, \x')}_{\mathbb{D}} +o\left( \eps^{\frac{2}{d+2}}\right) 
\end{align*}
follows. 
By using   \eqref{lemma:symetricForProof} and Lemma~\ref{LemmaBoundsUpper} (point  \ref{LemmaBoundsUpperPoint4}) again we get 
\begin{multline*}
    \mathbb{D}\leq \int C_\eps(\x') u_\eps(\x, \x') \diff (\rho_0\otimes \rho_0)(\x, \x')\\
    + w(\eps) \eps^{\frac{-d}{d+2}}\int_{C(\x')\geq \frac{1}{2}\|\x-\x'\|^2} C_\eps(\x') \diff (\rho_0\otimes \rho_0)(\x, \x') + o\left( \eps^{\frac{2}{d+2}}\right),
\end{multline*}
which yields 
$  \mathbb{D}\leq \int C_\eps(\x') \diff \rho_0(\x')+ o\left( \eps^{\frac{2}{d+2}}\right)$
and {\it a fortiori}
\begin{multline}\label{linearBound}
    \mathbb{L}_1 \leq    -\eps^{-1}\int_{(\Omega^{\delta'}_0)^2}  \left( C_\eps(\x') -\frac{1}{2}\|  \x-\x'\|_{\nabla^2 g^*(\x')}^2\right)_+^2 \diff (\rho_0\otimes \rho_0)(\x, \x')\\
    + \int C_\eps(\x') \diff \rho_0(\x')  +o\left( \eps^{\frac{2}{d+2}}\right)
\end{multline}

Now we deal with the penalty term. As before, we divide 
$$ 
\eps\ \left\|\frac{\diff \pi}{\diff (\rho_0\otimes\rho_1)}\right\|_{L^2(\rho_0\otimes\rho_1)}^2=\eps \int (v_\eps(\x, \x'))^2\diff (\rho_0\otimes \rho_0)(\x, \x') 
$$
into two additive terms
$$ \mathbb{Q}_1= \eps\int_{(\Omega^{\delta}_0\times \Omega_0) \cup ( \Omega_0\times \Omega^{\delta}_0)}  (u_\eps(\x, \x'))^2\diff (\rho_0\otimes \rho_0)(\x, \x') $$
and  
\begin{equation}
    \label{BoundQ2}
     \mathbb{Q}_2= \eps \int_{\Omega^{\delta,c}_0\times \Omega^{\delta,c}_0}  (h_{\eps, \delta}(\x, \x'))^2\diff (\rho_0\otimes \rho_0)(\x, \x')\leq   C{\eps^{\frac{2}{d+2}}} \ell_d(\Omega^{\delta,c}_0),
\end{equation}
where the bound is derived from Lemma~\ref{Lemma:BoundH}, and the constant \( C \) is independent of \( \delta \).
 We bound now $\mathbb{Q}_1$.  
 Since $C_\eps(\x)$ is H\"older continuous with constant $C\eps^{\frac{2}{d+2}}$ and owing to \eqref{stabilty}, it holds that 
 \begin{equation}
     \label{HolderC}
  \|C_\eps(\nabla \phi_\eps^*(\x)) - C_\eps(\x)\|\leq C \eps^{\frac{2}{d+2}} \omega(\eps) ,  \quad
 \text{with} \ \omega(\eps)\to 0.
 \end{equation}
  By 
Lemma~\ref{Lemma:developmentDiveps}, we also have 
$$ \left\vert  \| \nabla \phi_\eps^*(\x) -\nabla \phi_\eps^*(\x')\|_{    \nabla^2 g^*(\nabla \phi_\eps^*(\x'))}^2- \|  \x-\x'\|_{\nabla^2 g^*(\x')}^2 \right\vert \leq C \| \x-\x'\|^{2}( \| \x-\x'\|^{\alpha}+\omega(\eps)) $$
and 
$$ \left\vert  \| \nabla \phi_\eps^*(\x) -\nabla \phi_\eps^*(\x')\|_{    \nabla^2 g^*(\nabla \phi_\eps^*(\x))}^2- \|  \x-\x'\|_{\nabla^2 g^*(\x)}^2 \right\vert \leq C \| \x-\x'\|^{2}( \| \x-\x'\|^{\alpha}+\omega(\eps)). $$
 Since $\xi_{\eps}(\nabla \phi_\eps^*(\x), \nabla \phi_\eps^*(\x'))>0$ implies
 \begin{align*}
   C\eps^{\frac{1}{d+2}}& \geq  \|\nabla \phi_\eps^*(\x)-\nabla \phi_\eps^*(\x')\| \\
   &\geq \|\x-\x'\|-\|[\nabla \phi_\eps^*-{\rm Id}](\x)-[\nabla \phi_\eps^*-{\rm Id}](\x')\|\\
   &\geq \|\x-\x'\|\left(1-\|\nabla \phi_\eps^*-{\rm Id}\|_{\mathcal{C}^{1}(\Omega_0^\delta)}\right),
 \end{align*}

 were we used the lower bound on $\nabla^2 g^*(\x)$ and the fact that all norms on $\mathbb{R}^d$ are equivalent (which is hidden in the constant).
 Now, 
 using \eqref{stabilty} we get
 \begin{equation}
     \label{BoundnablaphiFinal1}
     \left\vert  \| \nabla \phi_\eps^*(\x) -\nabla \phi_\eps^*(\x')\|_{    \nabla^2 g^*(\nabla \phi_\eps^*(\x'))}^2- \|  \x-\x'\|_{\nabla^2 g^*(\x')}^2 \right\vert \leq C \omega(\eps) \eps^{\frac{2}{d+2}} 
 \end{equation}
and 
\begin{equation}
     \label{BoundnablaphiFinal2}
     \left\vert  \| \nabla \phi_\eps^*(\x) -\nabla \phi_\eps^*(\x')\|_{    \nabla^2 g^*(\nabla \phi_\eps^*(\x))}^2- \|  \x-\x'\|_{\nabla^2 g^*(\x)}^2 \right\vert \leq  C\omega(\eps) \eps^{\frac{2}{d+2}},
 \end{equation}
for all $(\x, \x')\in (\Omega^{\delta}_0\times \Omega_0) \cup ( \Omega_0\times \Omega^{\delta}_0) $ such that $\xi_{\eps}(\nabla \phi_\eps^*(\x), \nabla \phi_\eps^*(\x'))>0$ and
a function $\omega=\omega_\delta$ such that $\omega(\eps)\to 0$ as $\eps\to 0$.  Since $\nabla^2 g^* $ is uniformly continuous, 
by the same arguments as before, 
$$ \left\vert \|  \x-\x'\|_{\nabla^2 g^*(\x')}^2 - \|  \x-\x'\|_{\nabla^2 g^*(\x)}^2 \right\vert \leq \omega(\eps) \eps^{\frac{2}{d+2}},  $$
for all $(\x, \x')\in (\Omega^{\delta}_0\times \Omega_0) \cup ( \Omega_0\times \Omega^{\delta}_0) $ such that $\xi_{\eps}(\nabla \phi_\eps^*(\x), \nabla \phi_\eps^*(\x'))>0$. Therefore, this last fact combined with \eqref{HolderC}, \eqref{BoundnablaphiFinal1} and \eqref{BoundnablaphiFinal2} implies 
$$
    \xi_{\eps}(\nabla \phi_\eps^*(\x), \nabla \phi_\eps^*(\x')) 
    \leq  \frac{1}{ \eps}\bigg(C_\eps(\x)-\frac{1}{2}\|  \x-\x'\|_{\nabla^2 g^*(\x')}^2 
    +\omega(\eps) \eps^{\frac{2}{d+2}} \bigg)_+.  
$$
for some  function $\omega=\omega_\delta$ such that $\omega(\eps)\to 0$ as $\eps\to 0$. 
This combined with the point~\ref{LemmaBoundsUpperPoint4} of
Lemma~\ref{LemmaBoundsUpper} and \eqref{stabilty} yields
\begin{align*}
     u_\eps(\x, \x')&=\frac{\xi_{\eps}(\nabla \phi_\eps^*(\x), \nabla \phi_\eps^*(\x'))  } { \rho_\eps(\nabla \phi_\eps^*(\x))\rho_\eps(\nabla \phi_\eps^*(\x'))}\\
     & 
    \leq  \frac{1}{ \eps(1+o(1))}\bigg(C_\eps(\x)-\frac{1}{2}\|  \x-\x'\|_{\nabla^2 g^*(\x)}^2 
    +\omega(\eps) \eps^{\frac{2}{d+2}} \bigg)_+.  \\
     &\le\frac{1}{ \eps(1+o(1))}\bigg(C_\eps(\x)-\frac{1}{2}\|  \x-\x'\|_{\nabla^2 g^*(\x)}^2 
     \bigg)_+ +\omega(\eps) \eps^{-\frac{d}{d+2}},
\end{align*}
where the the error terms are uniform on the set of $(\x, \x')\in (\Omega^{\delta}_0\times \Omega_0) \cup ( \Omega_0\times \Omega^{\delta}_0) $ such that 
$
\xi_{\eps}(\nabla \phi_\eps^*(\x), \nabla \phi_\eps^*(\x'))>0$
and a fixed $0<\delta\leq \delta_0.
$ 
As a consequence,  it holds \textemdash up to $o(\eps^{\frac{2}{d+2}})$ additive terms, that
\begin{align*}
    \mathbb{Q}_1&= \eps\int_{(\Omega^{\delta}_0\times \Omega_0) \cup ( \Omega_0\times \Omega^{\delta}_0)}  (u_\eps(\x, \x'))^2\diff (\rho_0\otimes \rho_0)(\x, \x') \\
    &\leq \eps\int_{(\Omega^{\delta}_0\times \Omega_0) \cup ( \Omega_0\times \Omega^{\delta}_0)}  \bigg(C_\eps(\x)-\frac{1}{2}\|  \x-\x'\|_{\nabla^2 g^*(\x)}^2 
     \bigg)_+^2\diff (\rho_0\otimes \rho_0)(\x, \x')+ o(\eps^{\frac{2}{d+2}})\\
     &\leq \int_{(\Omega^{\delta'}_0)^2}  \bigg(C_\eps(\x)-\frac{1}{2}\|  \x-\x'\|_{\nabla^2 g^*(\x)}^2 
     \bigg)_+^2\diff (\rho_0\otimes \rho_0)(\x, \x')+ o(\eps^{\frac{2}{d+2}}),
\end{align*} 
for $0<\delta'<\delta$. 
The last display and \eqref{linearBound} yield 
\begin{multline*}
     2\, \mathbb{L}_1 +\mathbb{Q}_1 \leq    -\eps^{-1}\int_{(\Omega^{\delta'}_0)^2}  \left( C_\eps(\x') -\frac{1}{2}\|  \x-\x'\|_{\nabla^2 g^*(\x')}^2\right)_+^2 \diff (\rho_0\otimes \rho_0)(\x, \x')\\+
2\int C_\eps(\x') \diff \rho_0(\x') +o\left( \eps^{\frac{2}{d+2}}\right),
\end{multline*}
while  \eqref{eq:boundLinearBAd} and \eqref{BoundQ2} give 
$  2\, \mathbb{L}_2+\mathbb{Q}_2\leq C{\eps^{\frac{2}{d+2}}} \ell_d(\Omega^{\delta,c}_0). $
From the relation 
\[\mathbb{H}_\eps ( \pi_\eps)- \mathcal{W}_2^2(\rho_0, \rho_1)=2\, \mathbb{L}_1+2\, \mathbb{L}_1+\mathbb{Q}_1+\mathbb{Q}_2\] we get the estimate
\begin{multline*}
   \eps^{-\frac{2}{d+2}}\Big((\mathbb{H}_\eps ( \pi_\eps)- \mathcal{W}_2^2(\rho_0, \rho_1) \Big)\\
   \leq \underbrace{ 2\eps^{\frac{-2}{d+2}}\int  C_\eps(\x') \diff \rho_0(\x') - \eps^{\frac{d+4}{d+2}}\int_{(\Omega^{\delta}_0)^2} \bigg(C_\eps(\x)-\frac{1}{2}\|  \x-\x'\|_{\nabla^2 g^*(\x)}^2 
     \bigg)_+^2  \diff (\rho_0\otimes \rho_0)(\x, \x')}_{\text{from }2\mathbb{L}_1+\mathbb{Q}_1}\\
   + \underbrace{C \ell_d(\Omega^{\delta,c}_0)}_{\text{from }2\mathbb{L}_2+\mathbb{Q}_2}+o(1)
\end{multline*}
and we conclude as in the proof of the lower bound. 
\end{proof}
\subsection{General case} \label{section:General}
The aim of this section is to prove Theorem~\ref{Theorem:Main} in the general case, namely, to eliminate the assumption of convex support. To this end, we will construct the aforementioned {\it ``stained glass''} structure. Let
$ T_\delta=\{I_1\}_{i\in \N}  $ be a regular tessellation of squares $I_i$ of same length $ \beta>0 $ of $\R^d$.  Set 
$$\mathcal{G}(\beta)=\{ i\in \N: \ I_i\subset \Omega_0\}$$
and note that it is finite due to the compactness of $\Omega_0$. It is clear that we can write $\Omega_0$ as the union of  
$\Omega^{(\beta)}_0=  \bigcup_{i\in \mathcal{G}(\beta)} I_i $ and $\Omega^{(\beta,c)}_0=\Omega\setminus \bigcup_{i\in \mathcal{G}(\beta)} I_i$. We define the measure 
$$ \pi_{\beta, \eps}=\sum_{i\in \mathcal{G}(\beta)} \pi_{i}^{\eps} \chi_{\big(I_i\times \nabla g^* (I_i)\big)} + \gamma_{\beta, \eps} \chi_{\big(\Omega_0^{(\beta,c)}\times \Omega_0^{(\beta,c)}\big)}, $$
where
$ \pi_{i}^{\eps} \in \Pi(\chi_{I_i}\diff \rho_0, \chi_{[\nabla g^*(I_i)]} \diff \rho_1) $ is the true minimizer of 
$$ \frac{1}{2} \int \|\x-\y\|^2d\pi(\x,\y)+\frac{\varepsilon}{2}  \left\| \frac{d\pi}{  d(\rho_0 \otimes \rho_1) }\right\|_{L^2(\chi_{I_i}\diff \rho_0, \chi_{[\nabla g^*(I_i)]} \diff \rho_1)}^2,$$
among all $\pi\in \Pi(\chi_{I_i}\diff \rho_0, \chi_{[\nabla g^*(I_i)]} \diff \rho_1) $, and $ \gamma_{\beta, \eps} \in \Pi(\chi_{\Omega_0^{(\beta,c)}}\diff \rho_0, \chi_{\nabla^2 g^*(\Omega_0^{(\beta,c)})} \diff \rho_1) $ to be determined. By construction, 
$ \pi_{\beta, \eps}\in \Pi(\diff \rho_0, \diff \rho_1) .$ Moreover, since $\nabla g$ is the gradient of a convex function and  pushes  $ \chi_{[\nabla g^*(I_i)]}\diff \rho_1$ forward to $\chi_{[I_i]} \diff \rho_0$, it holds that 
$$ \mathcal{W}_2^2(\rho_0,\rho_1)=\mathcal{W}_2^2\left(\chi_{\left[\Omega_0^{(\beta,c)}\right]}\diff \rho_0,  \chi_{ \left[\nabla g^*\left(\Omega_0^{(\beta,c)}\right)\right]} \diff \rho_1\right)+\sum_{i\in \mathcal{G}(\beta)} \mathcal{W}_2^2\left(\chi_{[I_i]} \diff \rho_0, \chi_{[\nabla g^*(I_i)]}\diff \rho_1\right). $$
It also holds that 
\begin{multline*}
    \mathbb{H}_\eps ( \pi_{\beta, \eps})= \sum_{i\in \mathcal{G}(\beta)} \mathcal{T}_{2, \eps, (\cdot)^2}\left(\chi_{[I_i]} \diff \rho_0, \chi_{[\nabla g^*(I_i)]}\diff \rho_1\right)+ \\
    \int \|\x-\y\|^2 \diff \gamma_{\beta, \eps}(\x,\y) + \eps\cdot \left\|\frac{\diff \gamma_{\beta, \eps}}{\diff (\rho_0\otimes\rho_1)}\right\|_{L^2(\rho_0\otimes\rho_1)}^2.
\end{multline*}
We define \(\gamma_{\beta, \eps}\) following the same steps as \(\pi_\eps^{(2)}\) in the previous section. Note that  $\Omega_0^{(\beta,c)}$ satisfies the assumptions of    Lemma~\ref{Lemma:firstEstimates}.  Denote 
\[ 
W_\eps = \int \|\x - \y\|^2 \, \diff \gamma_{\beta, \eps}(\x, \y) + \eps\cdot \left\|\frac{\diff \gamma_{\beta, \eps}}{\diff (\rho_0 \otimes \rho_1)}\right\|_{L^2(\rho_0 \otimes \rho_1)}^2,
\]
and repeat the calculations yielding \eqref{eq:boundLinearBAd} and \eqref{BoundQ2}. As a consequence,
\[ 
\limsup_{\eps \to 0} \frac{W_\eps - \mathcal{W}_2^2\left(\chi_{[\Omega_0^{(\beta,c)}]} \diff \rho_0, \chi_{[\nabla g^* (\Omega_0^{(\beta,c)})]} \diff \rho_1\right)}{\eps^{\frac{2}{d+2}}} \leq C \ell_d \left(\Omega_0^{(\beta,c)}\right).
\]
Applying Lemma~\ref{Lemma:limsup} for each \(i \in \mathcal{G}(\beta)\) yields
\begin{align*}
    &\limsup_{\eps \to 0^+} \frac{\mathbb{H}_\eps (\pi_{\beta, \eps}) - \mathcal{W}_2^2(\rho_0, \rho_1)}{\eps^{\frac{2}{d+2}}} \\
    &  \qquad \leq\frac{d^{\frac{d+4}{d+2}}(d+2)^{\frac{2}{d+2}}}{\left(\mathcal{H}^{d-1}(\mathcal{S}^{d-1}) \right)^{\frac{2}{d+2}}}   \sum_{i \in \mathcal{G}(\beta)} \int_{I_i} \left(\rho_0(\x) \rho_1[\nabla g^*(\x)])\right)^{-\frac{1}{d+2}} \diff \rho_0(\x)+ C \ell_d \left(\Omega_0^{(\beta,c)}\right) \\
     & \qquad \leq\frac{d^{\frac{d+4}{d+2}}(d+2)^{\frac{2}{d+2}}}{\left(\mathcal{H}^{d-1}(\mathcal{S}^{d-1}) \right)^{\frac{2}{d+2}}}  \int_{\Omega \setminus \Omega_0^{(\beta,c)}} \left(\rho_0(\x) \rho_1[\nabla g^*(\x)])\right)^{-\frac{1}{d+2}} \diff \rho_0(\x) + C \ell_d \left(\Omega_0^{(\beta,c)}\right).
\end{align*} Theorem~\ref{Theorem:Main} follows by letting \(\beta \to 0\). \qed
\bibliography{biblio}

\begin{thebibliography}{}

\bibitem[\protect\citeauthoryear{Barenblatt}{Barenblatt}{1952}]{barenblatt_1952}
Barenblatt, G.~I. (1952).
\newblock On some unsteady fluid and gas motions in a porous medium.
\newblock {\em Prikladnaya Matematika i Mekhanika (in Russian)\/}~{\em
  10\/}(1), 67--78.

\bibitem[\protect\citeauthoryear{Blondel, Seguy, and Rolet}{Blondel
  et~al.}{2018}]{blondel2018smooth}
Blondel, M., V.~Seguy, and A.~Rolet (2018).
\newblock Smooth and sparse optimal transport.
\newblock In {\em International conference on artificial intelligence and
  statistics}, pp.\  880--889. PMLR.

\bibitem[\protect\citeauthoryear{Boussinesq}{Boussinesq}{1903}]{boussinesq_1903}
Boussinesq, J. (1903).
\newblock Recherches théoriques sur l'écoulement des nappes d'eau infiltrées
  dans le sol et sur le débit des sources.
\newblock {\em Journal de Mathématiques Pures et Appliquées\/}~{\em 10\/}(5),
  5--78.

\bibitem[\protect\citeauthoryear{Brenier}{Brenier}{1991}]{brenier1991polar}
Brenier, Y. (1991).
\newblock Polar factorization and monotone rearrangement of vector-valued
  functions.
\newblock {\em Communications on pure and applied mathematics\/}~{\em 44\/}(4),
  375--417.

\bibitem[\protect\citeauthoryear{Caffarelli}{Caffarelli}{1990}]{Caffarelly1990}
Caffarelli, L.~A. (1990).
\newblock A localization property of viscosity solutions to the monge-ampere
  equation and their strict convexity.
\newblock {\em Annals of Mathematics\/}~{\em 131\/}(1), 129--134.

\bibitem[\protect\citeauthoryear{Caffarelli}{Caffarelli}{1992}]{Caffarelli1992}
Caffarelli, L.~A. (1992).
\newblock The regularity of mappings with a convex potential.
\newblock {\em Journal of the American Mathematical Society\/}~{\em 5\/}(1),
  99--104.

\bibitem[\protect\citeauthoryear{Caffarelli}{Caffarelli}{1996}]{Caffarelli1996Bound2}
Caffarelli, L.~A. (1996, November).
\newblock Boundary regularity of maps with convex potentials--ii.
\newblock {\em The Annals of Mathematics\/}~{\em 144\/}(3), 453.

\bibitem[\protect\citeauthoryear{Carlier, Chernozhukov, and Galichon}{Carlier
  et~al.}{2016}]{Carlier2016}
Carlier, G., V.~Chernozhukov, and A.~Galichon (2016).
\newblock {Vector quantile regression: An optimal transport approach}.
\newblock {\em The Annals of Statistics\/}~{\em 44\/}(3), 1165 -- 1192.

\bibitem[\protect\citeauthoryear{Carrillo and Toscani}{Carrillo and
  Toscani}{2000}]{Carrillo_Toscani_2000_AsymptoticLO}
Carrillo, J.~A. and G.~Toscani (2000).
\newblock Asymptotic l1-decay of solutions of the porous medium equation to
  self-similarity.
\newblock {\em Indiana University Mathematics Journal\/}~{\em 49}, 0113--142.

\bibitem[\protect\citeauthoryear{Chewi, Niles-Weed, and Rigollet}{Chewi
  et~al.}{2024}]{chewi2024statistical}
Chewi, S., J.~Niles-Weed, and P.~Rigollet (2024).
\newblock Statistical optimal transport.
\newblock {\em arXiv preprint arXiv:2407.18163\/}.

\bibitem[\protect\citeauthoryear{Courty, Flamary, and Ducoffe}{Courty
  et~al.}{2018}]{courty}
Courty, N., R.~Flamary, and M.~Ducoffe (2018).
\newblock Learning {W}asserstein embeddings.
\newblock In {\em International Conference on Learning Representations}.

\bibitem[\protect\citeauthoryear{Cuesta and Matran}{Cuesta and
  Matran}{1989}]{CuestaMatran}
Cuesta, J.~A. and C.~Matran (1989).
\newblock {Notes on the Wasserstein Metric in Hilbert Spaces}.
\newblock {\em The Annals of Probability\/}~{\em 17\/}(3), 1264 -- 1276.

\bibitem[\protect\citeauthoryear{Cuturi}{Cuturi}{2013}]{cuturi2013sinkhorn}
Cuturi, M. (2013).
\newblock Sinkhorn distances: Lightspeed computation of optimal transport.
\newblock {\em Advances in neural information processing systems\/}~{\em 26}.

\bibitem[\protect\citeauthoryear{De~Lara, Gonz\'{a}lez-Sanz, and
  Loubes}{De~Lara et~al.}{2023}]{DeLaraDiffeo}
De~Lara, L., A.~Gonz\'{a}lez-Sanz, and J.-M. Loubes (2023).
\newblock Diffeomorphic registration using sinkhorn divergences.
\newblock {\em SIAM Journal on Imaging Sciences\/}~{\em 16\/}(1), 250--279.

\bibitem[\protect\citeauthoryear{del Barrio, Gonz{\'a}lez~Sanz, and Hallin}{del
  Barrio et~al.}{2024}]{del2024nonparametric}
del Barrio, E., A.~Gonz{\'a}lez~Sanz, and M.~Hallin (2024).
\newblock Nonparametric multiple-output center-outward quantile regression.
\newblock {\em Journal of the American Statistical
  Association\/}~(just-accepted), 1--43.

\bibitem[\protect\citeauthoryear{Eckstein and Nutz}{Eckstein and
  Nutz}{2023}]{Eckstein2023}
Eckstein, S. and M.~Nutz (2023, July).
\newblock Convergence rates for regularized optimal transport via quantization.
\newblock {\em Mathematics of Operations Research\/}.

\bibitem[\protect\citeauthoryear{Essid and Solomon}{Essid and
  Solomon}{2018}]{essid2018quadratically}
Essid, M. and J.~Solomon (2018).
\newblock Quadratically regularized optimal transport on graphs.
\newblock {\em SIAM Journal on Scientific Computing\/}~{\em 40\/}(4),
  A1961--A1986.

\bibitem[\protect\citeauthoryear{Figalli}{Figalli}{2017}]{Fi}
Figalli, A. (2017).
\newblock {\em The Monge-Amp\`ere Equation and Its Applications}.
\newblock Zurich: Zurich Lectures in Advanced Mathematics, European
  Mathematical Society (EMS).

\bibitem[\protect\citeauthoryear{Galichon}{Galichon}{2016}]{Galichon2016}
Galichon, A. (2016, September).
\newblock {\em Optimal Transport Methods in Economics}.
\newblock Princeton University Press.

\bibitem[\protect\citeauthoryear{Gonz{\'a}lez-Delgado, Gonz{\'a}lez-Sanz,
  Cort{\'e}s, and Neuvial}{Gonz{\'a}lez-Delgado et~al.}{2023}]{gonzalez2023two}
Gonz{\'a}lez-Delgado, J., A.~Gonz{\'a}lez-Sanz, J.~Cort{\'e}s, and P.~Neuvial
  (2023).
\newblock Two-sample goodness-of-fit tests on the flat torus based on
  {W}asserstein distance and their relevance to structural biology.
\newblock {\em Electronic Journal of Statistics\/}~{\em 17\/}(1), 1547--1586.

\bibitem[\protect\citeauthoryear{Gordaliza, del Barrio, Gamboa, and
  Loubes}{Gordaliza et~al.}{2019}]{gordaliza2019obtaining}
Gordaliza, P., E.~del Barrio, F.~Gamboa, and J.-M. Loubes (2019).
\newblock Obtaining fairness using optimal transport theory.
\newblock In {\em International Conference on Machine Learning}, pp.\
  2357--2365.

\bibitem[\protect\citeauthoryear{Hallin, del Barrio, Cuesta-Albertos, and
  Matrán}{Hallin et~al.}{2021}]{Hallin2020DistributionAQ}
Hallin, M., E.~del Barrio, J.~Cuesta-Albertos, and C.~Matrán (2021).
\newblock {Distribution and quantile functions, ranks and signs in dimension d:
  A measure transportation approach}.
\newblock {\em The Annals of Statistics\/}~{\em 49\/}(2), 1139 -- 1165.

\bibitem[\protect\citeauthoryear{Hundrieser, Mordant, Weitkamp, and
  Munk}{Hundrieser et~al.}{2023}]{hundrieser2023empirical}
Hundrieser, S., G.~Mordant, C.~A. Weitkamp, and A.~Munk (2023).
\newblock Empirical optimal transport under estimated costs: Distributional
  limits and statistical applications.
\newblock {\em arXiv preprint arXiv:2301.01287\/}.

\bibitem[\protect\citeauthoryear{Leibenzon}{Leibenzon}{1930}]{leibenzon_1930}
Leibenzon, L.~S. (1930).
\newblock The motion of a gas in a porous medium. complete works.
\newblock {\em Acad. Sci. URSS, Moscow\/}~{\em 2}.

\bibitem[\protect\citeauthoryear{Levy, Mohayaee, and von Hausegger}{Levy
  et~al.}{2021}]{Levy_2021}
Levy, B., R.~Mohayaee, and S.~von Hausegger (2021, June).
\newblock A fast semidiscrete optimal transport algorithm for a unique
  reconstruction of the early universe.
\newblock {\em Monthly Notices of the Royal Astronomical Society\/}~{\em
  506\/}(1), 1165–1185.

\bibitem[\protect\citeauthoryear{Manole, Balakrishnan, Niles-Weed, and
  Wasserman}{Manole et~al.}{2021}]{manole2021plugin}
Manole, T., S.~Balakrishnan, J.~Niles-Weed, and L.~Wasserman (2021).
\newblock Plugin estimation of smooth optimal transport maps.
\newblock {\em arXiv preprint arXiv:2107.12364\/}.

\bibitem[\protect\citeauthoryear{McCann}{McCann}{1995}]{mccann95}
McCann, R.~J. (1995).
\newblock {Existence and uniqueness of monotone measure-preserving maps}.
\newblock {\em Duke Mathematical Journal\/}~{\em 80\/}(2), 309 -- 323.

\bibitem[\protect\citeauthoryear{Monge}{Monge}{1781}]{monge1781memoire}
Monge, G. (1781).
\newblock M{\'e}moire sur la th{\'e}orie des d{\'e}blais et des remblais.
\newblock {\em Mem. Math. Phys. Acad. Royale Sci.\/}, 666--704.

\bibitem[\protect\citeauthoryear{Mordant}{Mordant}{2023}]{mordant2023regularised}
Mordant, G. (2023).
\newblock Regularised optimal self-transport is approximate {G}aussian mixture
  maximum likelihood.
\newblock {\em arXiv preprint arXiv:2310.14851\/}.

\bibitem[\protect\citeauthoryear{Mordant and Segers}{Mordant and
  Segers}{2022}]{mordant2022measuring}
Mordant, G. and J.~Segers (2022).
\newblock Measuring dependence between random vectors via optimal transport.
\newblock {\em Journal of Multivariate Analysis\/}~{\em 189}, 104912.

\bibitem[\protect\citeauthoryear{Muskat}{Muskat}{1937}]{muskat_1937}
Muskat, M. (1937).
\newblock The flow of homogeneous fluids through porous media.

\bibitem[\protect\citeauthoryear{Nutz}{Nutz}{2024}]{nutz2024quadratically}
Nutz, M. (2024).
\newblock Quadratically regularized optimal transport: Existence and
  multiplicity of potentials.

\bibitem[\protect\citeauthoryear{Otto}{Otto}{2001}]{Otto_2001}
Otto, F. (2001).
\newblock The geometry of dissipative evolution equations: The porous medium
  equation.
\newblock {\em Communications in Partial Differential Equations\/}~{\em
  26\/}(1-2), 101--174.

\bibitem[\protect\citeauthoryear{Pal}{Pal}{2019}]{pal2019difference}
Pal, S. (2019).
\newblock On the difference between entropic cost and the optimal transport
  cost.
\newblock {\em arXiv preprint arXiv:1905.12206\/}.

\bibitem[\protect\citeauthoryear{Peyré and Cuturi}{Peyré and
  Cuturi}{2019}]{MAL-073}
Peyré, G. and M.~Cuturi (2019).
\newblock Computational optimal transport: With applications to data science.
\newblock {\em Foundations and Trends® in Machine Learning\/}~{\em 11\/}(5-6),
  355--607.

\bibitem[\protect\citeauthoryear{Schiebinger, Shu, Tabaka, Cleary, Subramanian,
  Solomon, Gould, Liu, Lin, Berube, et~al.}{Schiebinger
  et~al.}{2019}]{schiebinger2019optimal}
Schiebinger, G., J.~Shu, M.~Tabaka, B.~Cleary, V.~Subramanian, A.~Solomon,
  J.~Gould, S.~Liu, S.~Lin, P.~Berube, et~al. (2019).
\newblock Optimal-transport analysis of single-cell gene expression identifies
  developmental trajectories in reprogramming.
\newblock {\em Cell\/}~{\em 176\/}(4), 928--943.

\bibitem[\protect\citeauthoryear{Sinkhorn}{Sinkhorn}{1964}]{sinkhorn1964relationship}
Sinkhorn, R. (1964).
\newblock A relationship between arbitrary positive matrices and doubly
  stochastic matrices.
\newblock {\em The annals of mathematical statistics\/}~{\em 35\/}(2),
  876--879.

\bibitem[\protect\citeauthoryear{Tameling, Stoldt, Stephan, Naas, Jakobs, and
  Munk}{Tameling et~al.}{2021}]{Tameling2021}
Tameling, C., S.~Stoldt, T.~Stephan, J.~Naas, S.~Jakobs, and A.~Munk (2021,
  March).
\newblock Colocalization for super-resolution microscopy via optimal transport.
\newblock {\em Nature Computational Science\/}~{\em 1\/}(3), 199–211.

\bibitem[\protect\citeauthoryear{Urbas}{Urbas}{1997}]{Urbas1997}
Urbas, J. (1997).
\newblock On the second boundary value problem for equations of
  {M}onge-{A}mpère type.
\newblock {\em Journal für die reine und angewandte Mathematik\/}~{\em 487},
  115--124.

\bibitem[\protect\citeauthoryear{Van~Assel, Vayer, Flamary, and
  Courty}{Van~Assel et~al.}{2023}]{van2023optimal}
Van~Assel, H., T.~Vayer, R.~Flamary, and N.~Courty (2023).
\newblock Optimal transport with adaptive regularisation.
\newblock {\em arXiv preprint arXiv:2310.02925\/}.

\bibitem[\protect\citeauthoryear{Villani}{Villani}{2003}]{Villani2003}
Villani, C. (2003).
\newblock {\em Topics in Optimal Transportation}.
\newblock Providence, Rhode Island: American mathematical societyr.

\bibitem[\protect\citeauthoryear{Villani}{Villani}{2008}]{Villani2008}
Villani, C. (2008).
\newblock {\em Optimal Transport: Old and New}.
\newblock Springer-Verlag Berlin Heidelberg.

\bibitem[\protect\citeauthoryear{Vázquez}{Vázquez}{2007}]{vazquez_2007}
Vázquez, J.~L. (2007).
\newblock The porous medium equation.
\newblock {\em Oxford Mathematical Monographs\/}.

\bibitem[\protect\citeauthoryear{Zeldovich and Kompaneets}{Zeldovich and
  Kompaneets}{1950}]{zeldovich_kompaneets_1903}
Zeldovich, Y.~B. and A.~S. Kompaneets (1950).
\newblock Towards a theory of heat conduction with thermal conductivity
  depending on the temperature.
\newblock {\em Collection of Papers Dedicated to 70th Anniversary of A. F.
  Ioffe. Izd. Akad. Nauk SSSR, Moscow\/}, 61--72.

\bibitem[\protect\citeauthoryear{Zhang, Mordant, Matsumoto, and
  Schiebinger}{Zhang et~al.}{2023}]{zhang2023manifold}
Zhang, S., G.~Mordant, T.~Matsumoto, and G.~Schiebinger (2023).
\newblock Manifold learning with sparse regularised optimal transport.
\newblock {\em arXiv preprint arXiv:2307.09816\/}.

\end{thebibliography}

\end{document}